\documentclass[a4paper,10pt]{article}

\usepackage{geometry}
\usepackage{amsmath}
\usepackage{amsfonts,color,amssymb,mathrsfs,stmaryrd}
\usepackage{amsthm}
\usepackage{amssymb}
\usepackage{url}
\usepackage{bera}
\usepackage[normalem]{ulem}
\usepackage{graphicx}
\usepackage{wrapfig}
\usepackage{mathrsfs}
\usepackage{bbm}
\usepackage{enumitem}
\usepackage{tikz}
\usetikzlibrary{decorations.pathreplacing}

\newcommand{\p}{\mathbb{P}}
\newcommand{\E}{\mathbb{E}}

\numberwithin{equation}{section}

\newtheorem{theorem}{Theorem}\numberwithin{theorem}{section}

\newtheorem{lemma}[theorem]{Lemma}

\newtheorem{assu}{Assumption}
\newtheorem{rema}[theorem]{Remark}

\def\E{\mathbb{E}}
\def\0{{\bf 0}}

\renewcommand{\E}{\mathbb E \,}

\newcommand{\Cov}{{\rm Cov}}
\newcommand{\Var}{{\rm Var}}

\def\beqn{\begin{equation}}
	\def\eeqn{\end{equation}}
\def\be{\begin{equation}}
	\def\ee{\end{equation}}

\def\qed{\hfill\hbox{${\vcenter{\vbox{
					\hrule height 0.4pt\hbox{\vrule width 0.4pt height 6pt
						\kern5pt\vrule width 0.4pt}\hrule height 0.4pt}}}$}}

\usepackage{titlesec}

\titleformat*{\section}{\normalfont\large\bfseries}
\titleformat*{\subsection}{\normalfont\bfseries}

\date{\vspace{-0.95cm}}

\begin{document}

	\title{Lower bounds for variances of Poisson functionals}
	
	\author{Matthias Schulte\footnotemark[1] \ and \ Vanessa Trapp\footnotemark[2]}
	
	\date{\today}
	\maketitle
	
	\footnotetext[1]{Hamburg University of Technology, Germany,
		matthias.schulte@tuhh.de}
	\footnotetext[2]{Hamburg University of Technology, Germany,
		vanessa.trapp@tuhh.de}

	\begin{abstract}
		Lower bounds for variances are often needed to derive central limit theorems. In this paper, we establish a lower bound for the variance of Poisson functionals that uses the difference operator of Malliavin calculus. 
		Poisson functionals, i.e.\ random variables that depend on a Poisson process, are frequently studied in stochastic geometry. We apply our lower variance bound to statistics of spatial random graphs, the $L^p$ surface area of random polytopes and the volume of excursion sets of Poisson shot noise processes. Thereby we do not only bound variances from below but also show positive definiteness of asymptotic covariance matrices and provide associated results on the multivariate normal approximation.
		\\\;\\
		\noindent
		\textbf{Keywords:} lower variance bounds, Poisson processes, covariance matrices, multivariate normal approximation, random polytopes, $L^p$ surface area, Poisson shot noise processes, spatial random graphs, Malliavin calculus\\
		\textbf{MSC 2020:} Primary: 60D05, Secondary: 60F05
	\end{abstract}

	\section{Introduction and main result}\label{sec:Intro}

As the variance quantifies the fluctuations of a random variable around its mean, upper bounds for variances are an important topic of probability theory. A main motivation to study lower bounds comes from the problem to establish central limit theorems. Here, after applying quantitative bounds for the normal approximation to standardised random variables, one has to divide by powers of the variance, whence it is essential to have lower bounds for the variance. In this paper, we derive such lower bounds for random variables that only depend on an underlying Poisson process. These so-called Poisson functionals play a crucial role in stochastic geometry but also appear in other branches of probability theory.

		Let $\eta$ be a Poisson process on a measurable space $(\mathbb{X},\mathcal{X})$ with a $\sigma$-finite intensity measure $\lambda$. The underlying probability space is denoted by $(\Omega,\mathscr{F},\mathbb{P})$. Let $\mathbf{N}$ denote the set of all $\sigma$-finite counting measures equipped with the $\sigma$-field generated by the mappings $\nu\mapsto\nu(B)$ for $B\in\mathcal{X}$. The Poisson process can be seen as a random element in $\mathbf{N}$. A detailed introduction to Poisson processes can be found in e.g.\ \cite{LP17}. A Poisson functional $F$ is a real-valued measurable function on $\Omega$ that can be written as $F=f(\eta)$, where $f$ is a real-valued measurable function on $\mathbf{N}$ and is called representative. For simplicity and by a slight abuse of notation, we denote a Poisson functional in the following by $F=F(\eta)$. If $F$ is square-integrable, we write $F\in L^2_\eta$.

Throughout this paper we are mostly interested in the asymptotic behaviour of Poisson functionals in two frameworks, namely increasing intensity or increasing observation window. More precisely, we study for $s\to\infty$ a family of Poisson functionals $F_s$, $s\geq1$, where $F_s$ is either a Poisson functional on a homogeneous Poisson process with intensity $s$ or a functional that considers only points of a fixed Poisson process in an observation window that extends to the full space for $s\to\infty$.

Central limit theorems for some Poisson functionals were established, for example, in \cite{AB93, BX06,BY05,CSY13,L19,LPY20,LSY19,LPS16,P05,PW08,PY01,R05,SY21}. Since the proofs require lower variance bounds as discussed above, these papers also study the asymptotic behaviour of the variance. Often convergence of the variance to a non-degenerate (i.e.\ non-zero) asymptotic variance constant is shown. Investigating the behaviour of the variance usually requires a lot of effort. This is the reason why we want to treat the problem of lower variance bounds as a separate issue from establishing central limit theorems in this paper. To this end, we provide a lower variance bound, which can be seen as the counterpart to the Poincar\'e inequality.

As mentioned above, a common problem is to show that the asymptotic variance constant is positive. But even if one has an explicit representation for the latter, it can be hard to show positivity because positive and negative terms could cancel out. Therefore, proving the non-degeneracy of the asymptotic variance can be a different problem than computing the limiting constant of the variance. In this case, it can be helpful to employ lower bounds for variances to deduce positivity of the asymptotic variance constant.

Since the covariance matrix $\Sigma_s\in\mathbb{R}^{m\times m}$ of Poisson functionals $F^{(1)}_s,\hdots, F^{(m)}_s$, $s\geq 1$, satisfies
$$
\Var\bigg[ \sum_{i=1}^m \alpha_i F_s^{(i)} \bigg] = \alpha^T \Sigma_s \alpha
$$
for all $\alpha=(\alpha_1,\hdots,\alpha_m)\in\mathbb{R}^m$, one can use lower bounds for variances to establish positive definiteness of the asymptotic covariance matrix $\Sigma=\lim_{s\to\infty} \Sigma_s$ if it exists. Knowing the positive definiteness of $\Sigma$ is of interest since it ensures that none of the Poisson functionals can be written asymptotically as a linear combination of the others. Furthermore, some bounds for the quantitative multivariate normal approximation (see e.g.\ \cite{SY21}) require the positive definiteness of the covariance matrix of the limiting normal distribution.

In order to present our main result, we need some notation and some further background on Poisson functionals.	For $x\in\mathbb{X}$ the difference operator of a Poisson functional $F=F(\eta)$ is defined by
	\begin{align*}
		D_xF=F(\eta+\delta_x)-F(\eta),
	\end{align*}
	where $\delta_x$ denotes the Dirac measure concentrated at $x$.	In general, the $n$-th iterated difference operator $D^n$ is recursively defined by
	\begin{align*}
	D^n_{x_1,\hdots,x_n}F=D_{x_1}(D^{n-1}_{x_2,\hdots,x_{n}}F)
	\end{align*} 
for $n>1$ and $x_1,\hdots,x_n\in\mathbb{X}$. In particular, for $x,y\in\mathbb{X}$ the iterated, second-order difference operator equals
	\begin{align*}
		D_{x,y}^2F=D_x(D_yF)=F(\eta+\delta_x+\delta_y)-F(\eta+\delta_x)-F(\eta+\delta_y)+F(\eta).
	\end{align*}
	For $F\in L_\eta^2$ define $f_n(x_1,\hdots,x_n)=\frac{1}{n!}\E[D^n_{x_1,\hdots,x_n}F]$ for $x_1,\hdots,x_n\in\mathbb{X}$ and $n\in\mathbb{N}$. Then, $f_n$ is symmetric and square-integrable for all $n\in\mathbb{N}$ and the Fock space representation of $F$ is given by
	\begin{align}
		\label{eq:fock_space}
		\E[F^2]=\E[F]^2+\sum_{n=1}^{\infty}n!\lVert f_n\rVert_n^2,
	\end{align}
	where $\lVert \cdot\rVert_n$ denotes the norm on $L^2(\lambda^n)$ (see, for example, \cite[Theorem 1.1]{LP11} or \cite[Theorem 18.6]{LP17}). Using this representation, one can directly derive
	\begin{align}\label{eqn:Bound_First_Chaos}
		\Var[F]=\sum_{n=1}^{\infty}n!\lVert f_n\rVert_n^2\geq \lVert f_1\rVert_1^2=\int(\E[D_xF])^2\;\mathrm{d}\lambda(x).
	\end{align}
	The problem with this lower variance bound is that the difference operator can in general be positive or negative and, thus, can have expectation zero. To overcome this issue, we provide in this paper a counterpart to the well-known Poincar\'{e} inequality 
	\begin{align}\label{eqn:Poincare}
		\Var[F]\leq \int\E[(D_xF)^2]\;\mathrm{d}\lambda(x)
	\end{align}
	for $F\in L^2_\eta$ (see, for example, \cite[Theorem 18.7]{LP17}). In the following main result we give a condition under which the variance of $F$ can be bounded from below by a constant times the right-hand side of the Poincar\'{e} inequality, whence we can think of it as a reversed Poincar\'e inequality. 
	
	\begin{theorem}
		\label{thm:varbound}
		Let $F\in L_\eta^2$ be a Poisson functional satisfying
		\begin{equation}
			\label{condition}
			\mathbb{E}\left[\int (D_{x,y}^2F)^2\; \mathrm{d}\lambda^2(x,y)\right]\leq\alpha \mathbb{E}\left[\int (D_{x}F)^2 \;\mathrm{d}\lambda(x)\right]<\infty
		\end{equation}
		for some constant $\alpha\geq 0$.
		Then
		\begin{equation}
			\label{prop}
			\mathrm{Var}[F]\geq\frac{4}{(\alpha+2)^2}\mathbb{E}\left[\int (D_{x}F)^2 \;\mathrm{d}\lambda(x)\right].
		\end{equation}
	\end{theorem}

The inequality \eqref{prop} provides a non-trivial lower bound for the variance as soon as one can show that the difference operator is non-zero with positive probability. To this end, one can construct special point configurations that lead to a non-zero difference operator and occur with positive probability. This is often much easier than to verify that the expectation of the difference operator is non-zero as required in \eqref{eqn:Bound_First_Chaos}.

Let us discuss some alternative approaches to derive lower variance bounds for Poisson functionals or statistics arising in stochastic geometry. In \cite[Theorem 5.2]{LPS16}, a general lower bound for variances of Poisson functionals is established, where, for fixed $k\in\mathbb{N}$ and $I_1,I_2\subseteq\{1,\hdots,k\}$, one has to bound
$$
\bigg|\mathbb{E}\bigg[ f\big(\eta+\sum_{i\in I_1} \delta_{x_i}\big) - f\big(\eta+\sum_{i\in I_2} \delta_{x_i}\big) \bigg] \bigg|
$$
from below for $x_1,\hdots,x_k\in\mathbb{X}$. Since here more than one point can be added, which allows to enforce particular point configurations, this expression is often easier to control than the expectation of the first difference operator in \eqref{eqn:Bound_First_Chaos}. But one still has the problem that the difference within the expectation can be both positive and negative.

In \cite{BY05,P05,PW08,PY01}, lower bounds for variances of so-called stabilising functionals of Poisson processes and sometimes also binomial point processes were deduced. These results have all in common that generalised difference or add-one-cost operators are required to be non-degenerate. This is similar to our work, but the random variable that has to be non-degenerate is more involved than the difference operator and, moreover, the results apply only to stabilising functionals and not to general Poisson functionals.

A further approach is to condition on some $\sigma$-field and to bound the variance from below by the expectation of the conditional variance with respect to this $\sigma$-field. In the context of stochastic geometry this was used, for example, in \cite{AB93} or \cite{BFV10, R05}. By conditioning on the $\sigma$-field it is sufficient to consider some particular point configurations similarly as in our Theorem \ref{thm:varbound}. In the recent preprint \cite{CX20}, a condition requiring that some conditional expectations are not degenerate is used to establish lower variance bounds for stabilising functionals.

In order to demonstrate how Theorem \ref{thm:varbound} can be applied, we derive lower variance bounds for specific examples from stochastic geometry:

\textbf{Spatial random graphs.} We consider degree and component counts of random geometric graphs and edge length functionals and degree counts of $k$-nearest neighbour graphs. By proving lower bounds for variances of linear combinations of such statistics, we show the positive definiteness of asymptotic covariance matrices. Combining these findings with the results from \cite[Section 3]{SY21} provides quantitative multivariate central limit theorems for the corresponding random vectors.

\textbf{Random polytopes.} By taking the convex hull of the points of a homogeneous Poisson process in the $d$-dimensional unit ball, one obtains a random polytope. We study the $L^p$ surface area, which generalises volume and surface area. For two different $L^p$ surface areas we show positive definiteness of the asymptotic covariance matrix and, as a consequence, a result for the multivariate normal approximation. In particular, this allows to study the joint behaviour of volume and surface area of the random polytope.

\textbf{Poisson shot noise processes.} We provide a lower variance bound for the volume of excursion sets of a Poisson shot noise process. In comparison to the works \cite{BST12}, \cite{L19} or \cite{LPY20} we modify the assumptions on the kernel function of the Poisson shot noise process.

The considered statistics of spatial random graphs fit into the framework of stabilising functionals of Poisson processes, whence the results for the non-degeneracy of the asymptotic variance of stabilising functionals discussed above might be applicable. The $L^p$ surface area is still stabilising, but here the variance does not scale like the intensity of the underlying Poisson process, whence the previously mentioned results are not available any more. Finally, in case of general Poisson shot noise processes we do not have stabilisation at all. In order to apply Theorem \ref{thm:varbound}, one has to bound the left-hand side of \eqref{condition} from above. In case of the spatial random graphs and the random polytope, this can be done easily by employing results from \cite{LSY19} due to stabilisation. 

This paper is organised as follows. Our main result Theorem \ref{thm:varbound} is proven in Section \ref{sec:proof_main_result}. The following three sections are devoted to applications, statistics of spatial random graphs in Section \ref{sec:spatial_geometric_graphs}, the $L^p$ surface area of random polytopes in Section \ref{sec:random_polytopes} and the excursion sets of a Poisson shot noise processes in Section \ref{sec:excursion_sets}. Finally, we recall some facts about stabilising functionals in the appendix.

	\section{Proof of Theorem \ref{thm:varbound}} \label{sec:proof_main_result}
	The proof of Theorem \ref{thm:varbound} relies upon using the Fock space representations of $F$ and its first two difference operators.
	
		\begin{proof}[Proof of Theorem \ref{thm:varbound}]
			For $n\in\mathbb{N}$ let $f_n$ denote the kernels of the Fock space representation of $F$. Recall that 
			\begin{align*}
				\mathrm{Var}[F]&=\sum_{n=1}^\infty n!\lVert f_n\rVert_n^2.
			\end{align*}
			
			First we assume $\alpha>0$. Then we know by assumption \eqref{condition} that $F,D_xF,D_{x,y}^2F\in L^2_\eta$ for $\lambda$-a.e.\ $x,y\in\mathbb{X}$. Using Fubini's theorem, the monotone convergence theorem and applying the Fock space representation \eqref{eq:fock_space} to the first and second order difference operator provides
			\allowdisplaybreaks
			\begin{align*}
				\mathbb{E}\left[\int (D_{x}F)^2 \;\mathrm{d}\lambda(x)\right]&=\int \sum_{n=0}^{\infty}\frac{1}{n!}\int\E[D_{x_1,\dots,x_n}^n(D_xF)]^2\;\mathrm{d}\lambda^n(x_1,\dots,x_n)\;\mathrm{d}\lambda(x)\\
				&= \sum_{n=0}^{\infty}\frac{1}{n!}\int\E[D_{x_1,\dots,x_n,x_{n+1}}^{n+1}F]^2\;\mathrm{d}\lambda^{n+1}(x_1,\dots,x_n,x_{n+1})\\
				&=\sum_{n=1}^\infty \frac{n}{n!}\int\E[D_{x_1,\dots,x_n}^{n}F]^2\;\mathrm{d}\lambda^n(x_1,\dots,x_n)\\
				&=\sum_{n=1}^{\infty}nn!\lVert f_n\rVert_n^2\\
				\text{and, similarly, }\;\;\;\;\;\;\;\;\;\;\;\;\;\;\;\;\;&\\
				\mathbb{E}\left[\int (D_{x,y}^2F)^2\; \mathrm{d}\lambda^2(x,y)\right]&=\int \sum_{n=0}^\infty\frac{1}{n!}\int\E[D_{x_1,\dots,x_n}^n(D_{x,y}F)]^2\;\mathrm{d}\lambda^{n}(x_1,\dots,x_{n})\;\mathrm{d}\lambda^2(x,y)\\
				&= \sum_{n=0}^{\infty}\frac{1}{n!}\int\E[D_{x_1,\dots,x_{n+2}}^{n+2}F]^2\;\mathrm{d}\lambda^{n+2}(x_1,\dots,x_{n+2})\\
				&=\sum_{n=2}^\infty \frac{n(n-1)}{n!}\int\E[D_{x_1,\dots,x_n}^{n}F]^2\;\mathrm{d}\lambda^n(x_1,\dots,x_n)\\
				&=\sum_{n=1}^{\infty}n(n-1)n!\lVert f_n\rVert_n^2.
			\end{align*}
			Therefore, assumption \eqref{condition} means that  $\sum_{n=1}^{\infty}n!n\lVert f_n\rVert_n^2(\alpha-n+1)\geq 0$. Additionally, $\left(n-\frac{(\alpha +2)}{2}\right)^2\geq 0$ implies $\frac{(\alpha+2)^2}{4}-n\geq n(\alpha-n+1)$ for any $n\in\mathbb{N}$. Thus, it holds
			\begin{align*}
				\frac{(\alpha+2)^2}{4}\mathrm{Var}[F]-\mathbb{E}\left[\int (D_{x}F)^2 \;\mathrm{d}\lambda(x)\right]&=\sum_{n=1}^{\infty}n!\lVert f_n\rVert_n^2\left(\frac{(\alpha+2)^2}{4}-n\right)\\&\geq \sum_{n=1}^{\infty}n!\lVert f_n\rVert_n^2n(\alpha-n+1)\geq 0,
			\end{align*}
			which provides the lower bound for the variance in \eqref{prop} for $\alpha>0$.
				
			For $\alpha=0$ we have that $D_{x,y}F=0$ almost surely for $\lambda$-a.e.\ $x,y\in\mathbb{X}$. Hence, all difference operators of order greater than or equal to $2$ vanish almost surely for $\lambda$-a.e.\ $x,y\in\mathbb{X}$. Therefore, $\lVert f_n\rVert_n=0$ for all $n\in\mathbb{N}$ with $n\geq 2$. It follows from the representation of the difference operator in terms of the kernels of the Fock space representation (see e.g. \cite[Theorem 3]{L16}) that $D_xF=f_1(x)$ almost surely for $\lambda$-a.e.\ $x\in\mathbb{X}$, which provides the bound in Theorem \ref{thm:varbound} for $\alpha=0$.
		\end{proof}

\begin{rema}
Note that Fock space representations also exist for functionals of isonormal Gaussian processes and for functionals of Rademacher sequences (i.e.\ sequences of independent random variables with values $\pm 1$). For these one can also define operators $D$ and $D^2$ whose Fock space representations are as in the Poisson case. Since our proof of Theorem \ref{thm:varbound} only requires the Fock space representations of $F$, $DF$ and $D^2F$, the statement of Theorem \ref{thm:varbound} continues to hold for functionals of isonormal Gaussian processes and for functionals of Rademacher sequences if we rewrite the integrals with respect to $\lambda$ in a proper way. For more details on the Fock space representations and the operators $D$ and $D^2$ we refer the reader to, for example, \cite{NP12} for the Gaussian case and \cite{KRT17} for the Rademacher case.
\end{rema}		
	
		\section{Spatial random graphs}\label{sec:spatial_geometric_graphs}
		
		In the following sections we apply our main result to problems from stochastic geometry. Therefore, we interpret Poisson processes as collections of random points in $\mathbb{X}$, which is why we write from now on for $A\subseteq\mathbb{X}$ under abuse of notation
		\begin{align*}
			\eta\cup A=\eta+\sum_{x\in A}\delta_x.
		\end{align*}
		Analogously, we use $\eta\cap A$ and $\eta\backslash A$.
		Throughout this paper, we denote by $\lambda_d$ the $d$-dimensional Lebesgue measure and by $\kappa_{d}$ the volume of the $d$-dimensional unit ball for $d\geq 1$. The $d$-dimensional closed ball with centre $x$ and radius $r$ is denoted by $B^d(x,r)$.
		
		Let $W\subset \mathbb{R}^d$ be a non-empty compact convex set with $\lambda_d(W)>0$. For $s\geq 1$ let $\eta_s$ be a homogeneous Poisson process on $W$ with intensity $s$, i.e.\ a Poisson process on $\mathbb{R}^d$ with intensity measure $\lambda=s\lambda_d|_{W}$, where $\lambda_d|_{W}$ denotes the restriction of the Lebesgue measure to $W$. In the following we study the asymptotic behaviour as $s\to\infty$.
		\subsection{Random geometric graph}
		
		 In this section we consider the vector of degree counts and the vector of component counts of a random geometric graph. For both examples we know from \cite[Section 3.2]{SY21} that, after centering and with a scaling of $s^{-1/2}$, they fulfil a quantitative central limit theorem in $d_2$- and $d_{convex}$- distance if the corresponding asymptotic covariance matrix is positive definite. In the following we show that the asymptotic covariance matrix is indeed positive definite.
		 
		 Let $G_{r_s}$ denote the random geometric graph that is generated by $\eta_s$ and has radius $r_s=\varrho s^{-1/d}$ for a fixed $\varrho>0$, i.e.\ the vertex set of the graph is $\eta_s$ and two distinct vertices $v_1, v_2\in \eta_{s}$ are connected by an edge if  $\lVert v_1-v_2\rVert\leq r_s$.
		 For $j\in\mathbb{N}_0$ let $V_j^{r_s}$ be the number of vertices of degree $j$ in $G_{r_s}$, i.e.\
		 \begin{align*}
		 	V_j^{r_s}=\sum_{y\in\eta_s}\mathbbm{1}\{\mathrm{deg}(y,\eta_s)=j\},
		 \end{align*}
		 where $\mathrm{deg}(y,\eta_s)$ stands for the degree of $y$ in $G_{r_s}$. Moreover, let $C_j^{r_s}$ denote the number of components of size $j$ in $G_{r_s}$, i.e.\
		 \begin{align*}
		 	C_j^{r_s}=\frac{1}{j}\sum_{y\in\eta_s}\mathbbm{1}\{ \lvert\mathrm{C}(y,\eta_s)\rvert=j\},
		 \end{align*}
		 where $\lvert\mathrm{C}(y,\eta_s)\rvert$ is the number of vertices of the component $C(y,\eta_s)$ of $y$ in $G_{r_s}$.
		\begin{theorem}
			\label{theorem:random_geometric_graph}
			\begin{enumerate}
				\item [a)]		For $s\to\infty$ the asymptotic covariance matrix of the vector of degree counts $\frac{1}{\sqrt{s}}(V_{j_1}^{r_s},\dots,V_{j_n}^{r_s})$ for distinct $j_i\in\mathbb{N}_0$, $i\in\{1,\dots,n\}$, is positive definite, i.e.\ for any $\alpha=(\alpha_1,\dots,\alpha_n)\in\mathbb{R}^n\backslash\{0\}$ there exists a constant $c>0$ such that for $s$ sufficiently large
				\begin{align*}
					\mathrm{Var}\left[\sum_{i=1}^{n}\alpha_iV_{j_i}^{r_s}\right]\geq cs.
				\end{align*}
				\item [b)]
				For $s\to\infty$ the asymptotic covariance matrix of the vector of component counts $\frac{1}{\sqrt{s}}(C_{j_1}^{r_s},\dots,C_{j_n}^{r_s})$ for distinct $j_i\in\mathbb{N}_0$, $i\in\{1,\dots,n\}$, is positive definite, i.e.\  for any $\alpha=(\alpha_1,\dots,\alpha_n)\in\mathbb{R}^n\backslash\{0\}$ there exists a constant $c>0$ such that for $s$ sufficiently large
				\begin{align*}
					\mathrm{Var}\left[\sum_{i=1}^{n}\alpha_iC_{j_i}^{r_s}\right]\geq cs.
				\end{align*}
			\end{enumerate}
		\end{theorem}
		Before we prove the theorem, we introduce the following lemma that provides condition \eqref{condition}. It gives an estimate for the expected integral of the squared second-order difference operator of a stabilising Poisson functional. We call a Poisson functional $F_s$ stabilising if it can be written as a sum of scores, i.e.\
		\begin{align}
			\label{eq:sum_of_scores}
			F_s=F_s(\eta_s)=\sum_{x\in\eta_s}\xi_s(x,\eta_s),
		\end{align} 
		where the scores $\xi_s$ are exponentially stabilising, fulfil a moment condition and decay exponentially fast with distance to a set $K$. For details on stabilising Poisson functionals and definitions see Section \ref{appendix:stabilising_functionals}. 
	
		\begin{lemma}
			\label{lemma:second_difference_operator}
			Let $F_s^{(1)},\dots,F_s^{(n)}$ be Poisson functionals on $\eta_s$, which can be written in the form of \eqref{eq:sum_of_scores} and whose corresponding scores $\xi_s^{(1)},\dots,\xi_s^{(n)}$ satisfy a $(4+p)$-th moment condition for $p>0$ and are exponentially stabilising.
			Then,  for any $\alpha=(\alpha_1,\dots,\alpha_n)\in\mathbb{R}^n\backslash\{0\}$ there exists a constant $c>0$ such that for $s\geq 1$,
			\begin{align*}
				\mathbb{E}\Big[\int_W\int_W \Big(\sum_{i=1}^{n}\alpha_iD_{x,y}^2F_{s}^{(i)}\Big)^2\;\mathrm{d}\lambda(x)\;\mathrm{d}\lambda(y)\Big]
				&\leq c s.
			\end{align*}
		\end{lemma}
	\begin{proof}
		We can apply \cite[Lemma 5.5 and Lemma 5.9]{LSY19}, i.e.\ for $i\in\{1,\dots,n\}$ and constants $\varepsilon\in(4,4+p)$, $\beta>0$ there exist constants $C_\varepsilon, C_\beta>0$ such that 
		\begin{align}
			\mathbb{E}\lvert D_{x}F_{s}^{(i)}(\eta_s\cup A)\rvert^{\varepsilon}\leq C_\varepsilon
			\label{lemma5.5}
		\end{align}
		for $A\subset W$ with $\lvert A\rvert\leq 1$, $x\in W$ and $s\geq 1$, where $\lvert A\rvert$ denotes the cardinality of $A$, and
		\begin{align}
			s\int_W \mathbb{P}(D_{x,y}^2F_{s}^{(i)}\neq 0)^\beta\;\mathrm{d}y\leq C_\beta
			\label{lemma5.9}
		\end{align}
		for $s\geq1$ and $x\in W$. Fix an $\varepsilon\in(4,4+p)$.
		Using \eqref{lemma5.5}, Hölder's inequality for $\frac{\varepsilon}{2}$ and $q=(1-\frac{2}{\varepsilon})^{-1}$ and Jensen's inequality provides
		\begin{align*}
			\mathbb{E}\lvert D_{x,y}^2F_{s}^{(i)}\rvert^2&=\mathbb{E}\left[\lvert D_{x,y}^2F_{s}^{(i)}\rvert^2\mathbbm{1}\{D_{x,y}^2F_{s}^{(i)}\neq 0\}\right]\\
			&\leq 	(\mathbb{E}\lvert D_{x,y}^2F_{s}^{(i)}\rvert^{\varepsilon})^{2/\varepsilon}\mathbb{P}(D_{x,y}^2F_{s}^{(i)}\neq 0)^{1/q}\\
			&=	(\mathbb{E}\lvert D_{x}F_{s}^{(i)}(\eta_s\cup\{y\})-D_{x}F_{s}^{(i)}(\eta_s)\rvert^{\varepsilon})^{2/\varepsilon}\mathbb{P}(D_{x,y}^2F_{s}^{(i)}\neq 0)^{1/q}\\
			&\leq 	\left(2^{\varepsilon-1}\left(\mathbb{E}\lvert D_{x}F_{s}^{(i)}(\eta_s\cup\{y\})\rvert^{\varepsilon}+\mathbb{E}\lvert D_{x}F_{s}^{(i)}(\eta_s)\rvert^{\varepsilon}\right)\right)^{2/\varepsilon}\mathbb{P}(D_{x,y}^2F_{s}^{(i)}\neq 0)^{1/q}\\
			&\leq 4C_\varepsilon^{2/\varepsilon}\mathbb{P}(D_{x,y}^2F_{s}^{(i)}\neq 0)^{1/q}
		\end{align*}
		for $i\in\{1,\dots,n\}$. Therefore, using Jensen's inequality and \eqref{lemma5.9}, it follows
		\allowdisplaybreaks
		\begin{align*}
			&\mathbb{E}\Big[\int_W\int_W \Big(D_{x,y}^2\sum_{i=1}^{n}\alpha_iF_{s}^{(i)}\Big)^2\;\mathrm{d}\lambda(x)\;\mathrm{d}\lambda(y)\Big]\\&
			\leq \int_W\int_W 	\mathbb{E}\Big[n\sum_{i=1}^{n}\alpha_i^2(D_{x,y}^2F_{s}^{(i)})^2\Big]\;\mathrm{d}\lambda(x)\;\mathrm{d}\lambda(y)\\&
			= n\sum_{i=1}^{n}\alpha_i^2\int_W\int_W 		\mathbb{E}\lvert D_{x,y}^2F_{s}^{(i)}\rvert^2\;\mathrm{d}\lambda(x)\;\mathrm{d}\lambda(y)\\&
			\leq n\sum_{i=1}^{n}\alpha_i^2 s\int_W s\int_W  4C_\varepsilon^{2/\varepsilon}\mathbb{P}(D_{x,y}^2F_{s}^{(i)}\neq 0)^{1/q}\;\mathrm{d}x\;\mathrm{d}y\\
			&\leq n\sum_{i=1}^{n}\alpha_i^2s\int_W 4C_\varepsilon^{2/\varepsilon}C_{1/q}\;\mathrm{d}x\leq cs
		\end{align*}
		for some constant $c>0$, which completes the proof.
	\end{proof}
		\begin{proof}[Proof of Theorem \ref{theorem:random_geometric_graph}]
			For $x\in W$ and $j\in\mathbb{N}_0$ the difference operators are given by
			\begin{align*}
				D_xV_j^{r_s}=\mathbbm{1}\{\mathrm{deg}(x,\eta_s\cup\{x\})=j\}+\sum_{y\in \eta_s}(\mathbbm{1}\{\mathrm{deg}(y,\eta_s\cup\{x\})=j\}-\mathbbm{1}\{\mathrm{deg}(y,\eta_s)=j\})
			\end{align*}
			and 
			\begin{align*}
				&D_xC_j^{r_s}=\frac{1}{j}\mathbbm{1}\{\lvert \mathrm{C}(x,\eta_s\cup\{x\})\rvert=j\}+\frac{1}{j}\sum_{y\in\eta_s}(\mathbbm{1}\{\lvert\mathrm{C}(y,\eta_s\cup\{x\})\rvert=j\}-\mathbbm{1}\{\lvert \mathrm{C}(y,\eta_s)\rvert=j\}).
			\end{align*}
			Let $m=\mathrm{argmax}_{i\in\{1,\dots,n\}:\alpha_i\neq 0}j_i$ and $x\in W$.
			For a) we consider configurations where 
			\begin{align*}
				\eta_s\left(B^d\left(x,\frac{r_s}{2}\right)\right)=j_m+1 \quad\text{ and }\quad \eta_s\Big(B^d\Big(x,\frac{3}{2}r_s\Big)\Big\backslash B^d\Big(x,\frac{r_s}{2}\Big)\Big)=0.
			\end{align*} 
			Then, it follows for any $y\in\eta_s$ with $y\in B^d(x,\frac{r_s}{2})\cap W$ that
			\begin{align*}
				\mathrm{deg}(y,\vartheta)=\begin{cases}
					j_m, &\text{ for }\vartheta=\eta_s,\\
					j_m+1, &\text{ for }\vartheta=\eta_s\cup\{x\}.
				\end{cases}
			\end{align*}
			The degrees of all the other points are not affected by adding $x$. Thus, in this situation only the numbers of points with degree $j_m$ and $j_m+1$ change. Due to the choice of $m$, we have
			\begin{align*}
				\left\lvert D_x\left(\sum_{i=1}^{n}\alpha_iV_{j_i}^{r_s}\right)\right\rvert=\left\lvert\sum_{i=1}^{n}\alpha_iD_xV_{j_i}^{r_s}\right\rvert=\lvert\alpha_mD_xV_{j_m}^{r_s}\rvert=\left\lvert\alpha_{m}(-(j_m+1))\right\rvert\geq \lvert\alpha_{m}\rvert.
			\end{align*}
			For b) we consider configurations where 
			\begin{align*}
				\eta_s\Big(B^d\Big(x,\frac{r_s}{2}\Big)\Big)=j_m \quad\text{ and }\quad \eta_s\Big(B^d\Big(x,\frac{3}{2}r_s\Big)\Big\backslash B^d\Big(x,\frac{r_s}{2}\Big)\Big)=0.
			\end{align*}
			It follows that $C_{j_m}^{r_s}$ decreases by $1$ by adding $x$ and $C_{j_m+1}^{r_s}$ increases by 1. The other component counts are not affected. Because of the choice of $m$, it holds
			\begin{align*}
				\left\lvert D_x\left(\sum_{i=1}^{n}\alpha_iC_{j_i}^{r_s}\right)\right\rvert=\left\lvert\sum_{i=1}^{n}\alpha_iD_xC_{j_i}^{r_s}\right\rvert=\lvert\alpha_mD_xC_{j_m}^{r_s}\rvert=\left\lvert\alpha_{m}\right\rvert.
			\end{align*}
			Let $A_s=\{x\in W:B^d(x,\frac{r_s}{2})\subset W\}$. Then, for $F_{j_i}^{r_s}=V_{j_i}^{r_s}$ or  $F_{j_i}^{r_s}=C_{j_i}^{r_s}$ for $i\in\{1,\dots, n\}$ and 
			$$
			k=\begin{cases}
				j_m+1, &\text{ for }F_{j_i}^{r_s}=V_{j_i}^{r_s},\\
				j_m, &\text{ for }F_{j_i}^{r_s}=C_{j_i}^{r_s},
			\end{cases}
			$$
			it follows for $s$ sufficiently large such that $\lambda_d(A_s)\geq\frac{\lambda_d(W)}{2}$ that
			\allowdisplaybreaks
			\begin{align*}
				&\mathbb{E}\int_W \Big(\sum_{i=1}^{n}\alpha_iD_xF_{j_i}^{r_s}\Big)^2\;\mathrm{d}\lambda(x)\geq s\alpha_{m}^2\int_W\mathbb{P}\Big(\Big\lvert \sum_{i=1}^{n}\alpha_iD_xF_{j_i}^{r_s}\Big\rvert\geq\rvert\alpha_{m}\rvert\Big)\;\mathrm{d}x\\
				&\geq s\alpha_{m}^2\int_{A_s}\mathbb{P}\Big(\eta_s\Big(B^d\Big(x,\frac{r_s}{2}\Big)\Big)=k,\eta_s\Big(B^d\Big(x,\frac{3}{2}r_s\Big)\Big\backslash B^d\Big(x,\frac{r_s}{2}\Big)\Big)=0\Big)\;\mathrm{d}x\\
				&\geq s\alpha_{m}^2\int_{A_s}\frac{(s\kappa_dr_s^d)^{k}}{2^{dk}k!}e^{-s\kappa_dr_s^d/2^d}e^{-s\kappa_d(3^d-1)r_s^d/2^d}\;\mathrm{d}x\\
				&\geq s\alpha_{m}^2\frac{\lambda_d(W)}{2}\frac{(\kappa_d\varrho^d)^{k}}{2^{dk}k!}e^{-\kappa_d3^d\varrho^d/2^d} =:c\cdot s,
			\end{align*}
			where $c>0$ depends on $W,\alpha,\varrho, k$ and $d$.
			
			Both functionals can be written as sums of scores as in \eqref{eq:sum_of_scores}. For $j\in\mathbb{N}_0$, $y\in \eta_s$ and $s\geq 1$ the score for the degree count of degree $j$ is given by
			\begin{align*}
				\xi_s(y,\eta_s)=\mathbbm{1}\{\mathrm{deg}(y,\eta_s)=j\}
			\end{align*}
			and for $j\in\mathbb{N}$, $y\in\eta_s$ and $s\geq 1$ the score for the number of components of size $j$ is
			\begin{align*}
					\xi_s(y,\eta_s)=\frac{1}{j}\mathbbm{1}\{\lvert \mathrm{C}(y,\eta_s)\rvert=j\}.
			\end{align*}
			These scores clearly fulfil a $(4+p)$-th moment condition and are 
			by \cite[proofs of Theorem 3.5 (b) and Theorem 3.6 (b)]{SY21} exponentially stabilising. Therefore, we can apply Lemma \ref{lemma:second_difference_operator}, which completes together with Theorem \ref{thm:varbound} the proof.
		\end{proof}
		
		\subsection{$k$-nearest neighbour graph}
			Central limit theorems for the total edge length of a $k$-nearest neighbour graph of a Poisson process are derived in e.g.\ \cite{AB93,BY05,LSY19,LPS16,PY01,PY05,SY21}. The first quantitative result can be found in \cite{AB93}. This convergence rate was further improved in \cite{PY05} before in \cite{LPS16} the presumably optimal rate was shown. In \cite{SY21} this result was transferred to the multivariate case of a vector of edge length functionals but it was left open to show in general that its asymptotic covariance matrix is positive definite. For edge length functionals of nonnegative powers this is proven in the following section.
			
		 We consider the $k$-nearest neighbour graph for $k\in\mathbb{N}$ that is generated by the Poisson process $\eta_s$, i.e.\ the undirected graph with vertex set $\eta_s$, where each vertex is connected with its $k$-nearest neighbours. The set of all $k$-nearest neighbours of $v_1\in\eta_s$ contains almost surely all $v_2\in\eta_s\backslash\{v_1\}$ for which $\lVert v_1-v_2\rVert\geq\rVert v_1-x\lVert$ for at most $k-1$ vertices $x\in\eta_s\backslash\{v_1\}$ or $\eta_{s}(B^d(v_1,\lVert v_1-v_2\rVert)\backslash\{v_1\})\leq k-1$.
		For $q\in[0,\infty)$ let $L_q$ denote the edge length functional of power $q$ of the $k$-nearest neighbour graph generated by $\eta_s$ which is defined by
	\begin{align*}
		L_q=\frac{1}{2}\sum_{(y,z)\in\eta_{s,\neq}^2}\mathbbm{1}\{z\in N(y,\eta_s) \text{ or }y\in N(z,\eta_s)\}\lVert y-z\rVert^q,
	\end{align*}
	where $\eta_{s,\neq}^2$ denotes the set of all pairs of disjoint points of $\eta_{s}$ and $N(y,\eta_s)$ is the set of all $k$-nearest neighbours of $y$ in the $k$-nearest neighbour graph generated by $\eta_s$. Let  $F_q=s^{q/d}L_q$ be its scaled version.
		\begin{theorem}
			\label{theorem:knn_edge_length}
			For $s\to\infty$ the asymptotic covariance matrix of $\frac{1}{\sqrt{s}}(F_{q_1},\dots,F_{q_n})$ for distinct $q_i\geq 0$, $i\in\{1,\dots,n\}$, is positive definite, i.e.\ for any $\alpha=(\alpha_1,\dots,\alpha_n)\in\mathbb{R}^n\backslash\{0\}$ there exists a constant $c>0$ such that for $s$ sufficiently large
			\begin{align*}
				\mathrm{Var}\left[\sum_{i=1}^n\alpha_iF_{q_i}\right]\geq cs.
			\end{align*}
		\end{theorem} In order to prove this theorem, we need the following lemma, which considers a slightly more general situation since it will be also employed in a further proof.
		\begin{lemma}
			\label{lemma_By}
			Let $k\in\mathbb{N}$ and $j\geq 1$ be fixed. Then there exist constants $c_1,c_2>0$ depending on $k,j,d$ and $W$ such that for all $\varepsilon>0$ and $x\in W$ with $B^d(x,2(j+1)\varepsilon)\subset W$,
			\begin{align*}
				\mathbb{P}(\exists y\in \eta_s\backslash B^d(x, j\varepsilon):\eta_s(A_{j,\varepsilon}(x,y))\leq k-1)\leq c_1e^{-sc_2\varepsilon^d},
			\end{align*}
		where $A_{j,\varepsilon}(x,y)=(B^d(y,\lVert x-y\rVert-(j-1)\varepsilon)\cap W)\backslash (B^d(x,j\varepsilon)\cup\{y\})$.
		\end{lemma}
		\begin{proof}
			Let $x\in W$ with $B^d(x,2(j+1)\varepsilon)\subset W$. Then, for $y\in W$ with $j\varepsilon< \lVert x-y\rVert\leq(j+1)\varepsilon$ we have that $B^d(y,\rVert x-y\lVert)\subset W$. Therefore, since $y\notin B^d(x,j\varepsilon)$,
			\begin{align*}
				\lambda_d(A_{j,\varepsilon}(x,y))\geq \frac{1}{2}\kappa_{d}(\lVert x-y\rVert-(j-1)\varepsilon)^d. 
			\end{align*}
			For $y\in W$ with $\lVert x-y\rVert> (j+1)\varepsilon$ it holds that $\lVert x-y\rVert-j\varepsilon\geq \frac{1}{2}(\lVert x-y\rVert -(j-1)\varepsilon)$. Moreover, $(B^d(y,\lVert x-y\rVert-j\varepsilon)\cap W)\backslash\{y\}\subset A_{j,\varepsilon}(x,y)$. Hence, with \cite[Lemma 7.4]{LPS16} there is a constant $c_W>0$ only depending on $W$ such that
			\begin{align*}
				\lambda_d(A_{j,\varepsilon}(x,y))&\geq \lambda_d(B^d(y,\lVert x-y\rVert-j\varepsilon)\cap W)\geq c_W(\lVert x-y\rVert-j\varepsilon)^d\\&\geq c_W\Big(\frac{1}{2}(\lVert x-y\rVert -(j-1)\varepsilon)\Big)^d.
			\end{align*}
			Altogether, for $y\in W\backslash B^d(x,j\varepsilon)$ it follows
			\begin{align*}
				&\lambda_d(A_{j,\varepsilon}(x,y))\geq c(\lVert x-y\rVert-(j-1)\varepsilon)^d
			\end{align*}
			for some constant $c>0$.
			For $t\in\mathbb{N}_0$ there exist constants $\tilde{c}_1,\tilde{c}_2>0$ such that $z^te^{-z}\leq \tilde{c}_1e^{-\tilde{c}_2z}$ for all $z>0$.
			Hence, using the Mecke formula and spherical coordinates, we get 
			\begin{align*}
			&\mathbb{P}(\exists y\in \eta_s\backslash B^d(x, j\varepsilon):\eta_s(A_{j,\varepsilon}(x,y))\leq k-1)\\
			&\leq \mathbb{E}\Bigg[\sum_{y\in\eta_{s}\backslash B^d(x, j\varepsilon)}\mathbbm{1}\{\eta_s(A_{j,\varepsilon}(x,y))\leq k-1\}\Bigg]\\
			&\leq s\int_{\mathbb{R}^d\backslash B^d(x,j\varepsilon)}\mathbb{P}(\eta_s(A_{j,\varepsilon}(x,y))\leq k-1)\;\mathrm{d}y\\
			&=s\int_{\mathbb{R}^d\backslash B^d(x,j\varepsilon)}\sum_{i=0}^{k-1}\frac{\lambda(A_{j,\varepsilon}(x,y))^i}{i!}e^{-\lambda(A_{j,\varepsilon}(x,y))}\;\mathrm{d}y\\
			&\leq s\int_{\mathbb{R}^d\backslash B^d(x,j\varepsilon)}\hat{c}_1e^{-\hat{c}_2\lambda(A_{j,\varepsilon}(x,y))}\;\mathrm{d}y\\
			&\leq  \int_{\mathbb{R}^d\backslash B^d(x,j\varepsilon)}\hat{c}_1se^{-\hat{c}_2sc(\lVert x-y\rVert-(j-1)\varepsilon)^d}\;\mathrm{d}y\\
			&=d\kappa_d\int_{\varepsilon}^\infty \hat{c}_1s(r+(j-1)\varepsilon)^{d-1}e^{-\hat{c}_2scr^d}\;\mathrm{d}r\leq c_1e^{-sc_2\varepsilon^d}
			\end{align*}
			for suitable constants $\hat{c}_1,\hat{c}_2,c_1,c_2>0$.
		\end{proof}
		\begin{proof}[Proof of Theorem \ref{theorem:knn_edge_length}]
			Let $e_i$ denote the $d$-dimensional standard unit vector in the $i$-th direction. For $\varepsilon>0$, $x\in W$ with $B^d(x,4\varepsilon)\subset W$ and $\hat{x}=x+\frac{3}{4}\varepsilon e_1$, we consider configurations where $\eta_s(B^d(\hat{x},\varepsilon/4))=k+1$, $\eta_s(B^d(x,\varepsilon)\backslash B^d(\hat{x},\varepsilon/4))=0$ and $\eta_s(A_{1,\varepsilon}(x,y))\geq k$ for all $y\in\eta_s\backslash B^d(x,\varepsilon)$, where $A_{1,\varepsilon}(x,y)$ is defined as in Lemma \ref{lemma_By}. Then, for $q\geq 0$ the difference operator of $F_q$ is given by
			\begin{align*}
				D_xF_q&=s^{q/d}\sum_{y\in N(x,\eta_s\cup\{x\})}\lVert x-y\rVert^q.
			\end{align*}
			Inserting $j=1$ in Lemma \ref{lemma_By} provides
		\begin{align*}
			\mathbb{P}(\exists y\in \eta_s\backslash B^d(x, \varepsilon):\eta_s(A_{1,\varepsilon}(x,y))\leq k-1)\leq c_1e^{-sc_2\varepsilon^d}
		\end{align*}
		for some constants $c_1,c_2>0$.
		
			Now, let $m=\mathrm{argmax}_{i\in\{1,\dots,n\}:\alpha_i\neq 0}q_i$ and assume without loss of generality $\alpha_m>0$.
			If $\alpha_{i}\geq 0$ for all $i\in\{1,\dots,n\}$, we choose  $\varepsilon=\bar{c}s^{-1/d}$ with $\bar{c}\geq 1$ large enough such that we have for the configurations mentioned above
			\begin{align*}
				D_x\sum_{i=1}^n\alpha_iF_{q_i}\geq \alpha_ms^{q_m/d}\sum_{y\in N(x,\eta_s\cup\{x\})}\lVert x-y\rVert^{q_m}\geq \alpha_mk\left(\frac{s^{1/d}\varepsilon}{2}\right)^{q_m}\geq 1
			\end{align*}
			and $c_1e^{-sc_2\varepsilon^d}<\frac{1}{2}$.
			Otherwise, let $\ell=\mathrm{argmax}_{i\in\{1,\dots,n\}:\alpha_i<0}q_i$. Then, $q_m>q_\ell$ and it follows for the configurations mentioned above for $s^{1/d}\varepsilon\geq1$,
			\begin{align*}
				&D_x\sum_{i=1}^n\alpha_iF_{q_i}=\sum_{i=1}^n\alpha_is^{q_i/d}\sum_{y\in N(x,\eta_s\cup\{x\})}\lVert x-y\rVert^{q_i}\\
				&\geq \alpha_ms^{q_m/d}\sum_{y\in N(x,\eta_s\cup\{x\})}\lVert x-y\rVert^{q_m}-\sum\limits_{\substack{i\in \{1,\dots,n\}: \\ \alpha_i<0}}(-\alpha_i)s^{q_i/d}\sum_{y\in N(x,\eta_s\cup\{x\})}\lVert x-y\rVert^{q_i}\\
				&\geq \alpha_m\sum_{y\in N(x,\eta_s\cup\{x\})}\left(s^{1/d}\lVert x-y\rVert\right)^{q_m}-\sum\limits_{\substack{i\in \{1,\dots,n\}: \\ \alpha_i<0}}(-\alpha_i)\sum_{y\in N(x,\eta_s\cup\{x\})}(s^{1/d}\varepsilon)^{q_i}\\
				&\geq \alpha_mk\left(\frac{s^{1/d}\varepsilon}{2}\right)^{q_m}-\sum\limits_{\substack{i\in \{1,\dots,n\}: \\ \alpha_i<0}}(-\alpha_i)k(s^{1/d}\varepsilon)^{q_\ell}\\
				&\geq k(s^{1/d}\varepsilon)^{q_\ell}\Bigg(\alpha_m\frac{1}{2^{q_m}}\left(s^{1/d}\varepsilon\right)^{q_m-q_\ell}-\sum\limits_{\substack{i\in \{1,\dots,n\}: \\ \alpha_i<0}}(-\alpha_i)\Bigg).
			\end{align*}
			In this case, choose $\varepsilon= s^{-1/d}\bar{c}>0$ with $\bar{c}\geq 1$ large enough such that $c_1e^{-sc_2\varepsilon^d}<\frac{1}{2}$ and
			\begin{align*}
				D_x\sum_{i=1}^n\alpha_iF_{q_i}\geq\alpha_m\frac{1}{2^{q_m}}\left(s^{1/d}\varepsilon\right)^{q_m-q_\ell}-\sum\limits_{\substack{i\in \{1,\dots,n\}: \\ \alpha_i<0}}(-\alpha_i)\geq 1.
			\end{align*} 
			Let $A_s=\{x\in W:B^d(x,4\varepsilon)\subset W\}$. Due to the independence of $\eta_{s}(B^d(\hat{x},\varepsilon/4))$, $\eta_{s}(B^d(x,\varepsilon)\backslash B^d(\hat{x},\varepsilon/4))$ and $\eta_{s}(A_{1,\varepsilon}(x,y))$ for $y\in\eta_s\backslash B^d(x,\varepsilon)$ and $x\in A_s$ and by Lemma \ref{lemma_By} we have for $s$ large enough such that $\lambda_d(A_s)\geq\frac{\lambda_d(W)}{2}$,
			\begin{align*}
				&\mathbb{E}\Big[\int_W\Big(D_x\sum_{i=1}^n\alpha_iF_{q_i}\Big)^2\;\mathrm{d}\lambda(x)\Big]\geq s	\int_W\mathbb{P}\Big(D_x\sum_{i=1}^n\alpha_iF_{q_i}\geq 1 \Big)\;\mathrm{d}x\\
				&\geq  s	\int_W\mathbb{P}\big(\eta_s(B^d(\hat{x},\varepsilon/4))=k+1, \eta_s(B^d(x,\varepsilon)\backslash B^d(\hat{x},\varepsilon/4))=0,\\&\;\;\;\;\;\;\;\;\;\;\;\;\;\;\;\;\;\eta_s(A_{1,\varepsilon}(x,y))\geq k \;\forall y\in\eta_s\backslash B^d(x,\varepsilon)\big)\;\mathrm{d}x\\
				&\geq s	\int_{A_s}\frac{(s\kappa_d\varepsilon^d)^{k+1}}{4^{d(k+1)}(k+1)!}e^{-s\kappa_d\varepsilon^d/4^d}e^{-s\kappa_d\varepsilon^d(1-1/4^d)}(1-c_1e^{-sc_2\varepsilon^d})\;\mathrm{d}x\\
				&\geq  s	\frac{\lambda_d(W)}{2}\frac{(\kappa_d\bar{c}^d)^{k+1}}{4^{d(k+1)}(k+1)!}e^{-\kappa_d\bar{c}^d}\cdot\frac{1}{2} =:c_{q,\alpha,k,W,d}s.
			\end{align*}
			Our functionals can be written as sums of scores as in \eqref{eq:sum_of_scores}. For $y\in\eta_s$, $q\geq0$ and $s\geq 1$ the corresponding score of $F_q$ is given by
			\begin{align*}
				\xi_s(y,\eta_s)=\sum_{z\in N(y,\eta_s)}\mathbbm{1}\{y\in N(z,\eta_s)\}\frac{\Vert y-z\rVert^q}{2}+\mathbbm{1}\{y\notin N(z,\eta_s)\}\Vert y-z\rVert^q.
			\end{align*} 
			The scores $(\xi_s)_{s\geq 1}$ fulfil a $(4+p)$-th moment condition (see the proof of \cite[Theorem 3.1]{LSY19}) and are
			by \cite[proof of Theorem 3.1]{SY21} exponentially stabilising. Therefore, we can apply Lemma \ref{lemma:second_difference_operator}, which completes together with Theorem \ref{thm:varbound} the proof.
		\end{proof}
		In the following we consider a second statistic of $k$-nearest neighbour graphs, namely the number of vertices with a given degree. Similarly to the previous example, it was shown in \cite[Theorem 3.3]{SY21} that a vector of these degree counts fulfils a quantitative multivariate central limit theorem in $d_2$- and $d_{convex}$-distance if its asymptotic covariance matrix is positive definite. 
		
		For $j\in\mathbb{N}_0$ let $V_j^{k}$ denote the number of vertices of degree $j$ in the $k$-nearest neighbour graph generated by $\eta_s$, i.e.\
	\begin{align*}
		V_j^{k}=\sum_{y\in\eta_s}\mathbbm{1}\{\mathrm{deg}(y,\eta_s)=j\}.
	\end{align*}
		We study the vector $(V_{j_1}^k,\dots, V_{j_n}^k)$ for distinct $j_i\geq k$, $i\in\{1,\dots,n\}$.  By \cite[Lemma 8.4]{Y98} the vertices of a $k$-nearest neighbour graph have bounded degree. Therefore, we consider $j_i\in \{k,k+1,\dots,k_{\mathrm{max}}\}$ for $i\in\{1,\dots,n\}$, where $k_{\mathrm{max}}$ denotes the maximal possible degree that occurs with a positive probability.
		\begin{theorem}
		
			For $d\geq 2$, $n\leq k_\mathrm{max}-k+1$ and $s\to\infty$ the asymptotic covariance matrix of $\frac{1}{\sqrt{s}}(V_{j_1}^k,\dots, V_{j_n}^k)$ for distinct $j_i\in \{k,k+1,\dots,k_{\mathrm{max}}\}$, $i\in\{1,\dots,n\}$, is positive definite, i.e.\ for any $\alpha=(\alpha_1,\dots\alpha_n)\in\mathbb{R}^n\backslash\{0\}$ there exists a constant $c>0$ such that for $s$ sufficiently large
			\begin{align*}
				\mathrm{Var}\left[\sum_{i=1}^{n}\alpha_iV_{j_i}^{k}\right]\geq cs.
			\end{align*}
			
		\end{theorem}
		\begin{proof}
			First note that the degrees $j_1,\dots,j_n$ are chosen in such a way that they can occur in a $k$-nearest neighbour graph. A vertex can have $k$ neighbours if it is only connected to its $k$ nearest neighbours and can have up to $k_{\mathrm{max}}$ neighbours by the definition of $k_{\mathrm{max}}$. All degrees in between can occur as well as can be seen from the following construction. Assume we have a configuration where $x$ has $k_{\mathrm{max}}$ neighbours. Then we delete $1\leq t\leq k_{\mathrm{max}}-k$ vertices which are connected to $x$ but are not one of the $k$ nearest neighbours of $x$ and all other vertices that are not connected to $x$. Consequently, we obtain a configuration where $x$ has degree $k_{\mathrm{max}}-t$.
			This means that $\mathbb{P}(\mathrm{deg}(x,\beta_{j_i}\cup\{x\})=j_i)>0$ for $i\in\{1,\dots,n\}$, where $\beta_{j_i}$ denotes a binomial point process of $j_i$ independent random points uniformly distributed in $B^d(0,1)$. Obviously, these probabilities do not change if we take a binomial point process on any other ball.
			
			The difference operator of $V_j^{k}$ is given by
			\begin{align*}
				D_xV_j^{k}=\mathbbm{1}\{\mathrm{deg}(x,\eta_s\cup\{x\})=j\}+\sum_{y\in\eta_s}(\mathbbm{1}\{\mathrm{deg}(y,\eta_s\cup\{x\})=j\}-\mathbbm{1}\{\mathrm{deg}(y,\eta_s)=j\})
			\end{align*}
			for $x\in W$.
			Denote $I=\{i\in\{1,\dots,n\}:\alpha_i\neq0\}$ and $m=\mathrm{argmin}_{i\in I}j_i$. We can assume $\alpha_m>0$ without loss of generality. In the following we distinct several cases that are illustrated in Figure \ref{fig:kNN}.
			
			\medskip
			
			\noindent \textit{Case 1: $j_m>k$}\\
			Let $\varepsilon>0$ and $x\in W$ with $B^d(x, 8\varepsilon)\subset W$.
			We consider configurations where $\eta_s(B^d(x,\varepsilon))=j_m,\eta_s(B^d(x,3\varepsilon)\backslash B^d(x,\varepsilon))=0$ and $\eta_s(A_{3,\varepsilon}(x,y))\geq k$ for all $y\in\eta_s\backslash B^d(x,3\varepsilon)$ with $A_{3,\varepsilon}(x,y)$ defined as in Lemma \ref{lemma_By}. Applying
			Lemma \ref{lemma_By} for $j=3$ provides
			\begin{align*}
				\mathbb{P}(\exists y\in \eta_s\backslash B^d(x, 3\varepsilon):\eta_s(A_{3,\varepsilon}(x,y))\leq k-1)
				&\leq c_1e^{-sc_2\varepsilon^d}.
			\end{align*}
			Now, choose $\varepsilon=\bar{c} s^{-1/d}>0$ for $\bar{c}>1$ such that $c_1e^{-sc_2\varepsilon^d}\leq\frac{1}{2}$.
			Then, $x$ is connected to all $z\in \eta_s\cap B^d(x,\varepsilon)$ and we have
			\begin{align*}
				D_x\sum_{i=1}^n\alpha_iV_{j_i}^k\geq\alpha_m.
			\end{align*}
			Let $A_s=\{x\in W:B^d(x,8\varepsilon)\subset W\}$ and $s$ large enough such that $\lambda_d(A_s)>\frac{\lambda_d(W)}{2}$. Then, using independence properties we have for $p_m=\mathbb{P}\left(\mathrm{deg}(x,\beta_{j_m}\cup\{x\})=j_m\right)>0$,
			\allowdisplaybreaks
			\begin{align*}
				&\mathbb{E}\Big[\int_W \Big(D_x\sum_{i=1}^{n}\alpha_iV_{j_i}^{k}\Big)^2\;\mathrm{d}\lambda(x)\Big]\geq \alpha_m^2\int_W\mathbb{P}\Big(D_x\sum_{i=1}^{n}\alpha_iV_{j_i}^{k}\geq\alpha_m\Big)\;\mathrm{d}\lambda(x)\\
				&\geq \alpha_m^2\int_{A_s}\mathbb{P}\left(\eta_s(B^d(x,\varepsilon))=j_m,\eta_s(B^d(x,3\varepsilon)\backslash B^d(x,\varepsilon))=0,\mathrm{deg}(x,\eta_s|_{B^d(x,\varepsilon)}\cup\{x\})=j_m\right)\\
				&\;\;\;\;\;\;\;\;\;\;\;\;\;\cdot\mathbb{P}\left(\eta_s(A_{3,\varepsilon}(x,y))\geq k\; \forall y\in\eta_s\backslash B^d(x,3\varepsilon)\right)\;\mathrm{d}\lambda(x)\\
				&\geq s\frac{\alpha_m^2}{2} \int_{A_s}\mathbb{P}\left(\eta_s(B^d(x,\varepsilon))=j_m,\eta_s(B^d(x,3\varepsilon)\backslash B^d(x,\varepsilon))=0\right)\\&\;\;\;\;\;\;\;\;\;\;\;\;\;\;\;\;\;\;\;\cdot\mathbb{P}\left(\mathrm{deg}(x,\eta_s|_{B^d(x,\varepsilon)}\cup\{x\})=j_m|\eta_s(B^d(x,\varepsilon))=j_m\right)\;\mathrm{d}x\\
				&=s\frac{\alpha_m^2}{2}\int_{A_s}\frac{(s\kappa_d\varepsilon^d)^{j_m}}{j_m!}e^{-s\kappa_d\varepsilon^d}e^{-s\kappa_d(3^d-1)\varepsilon^d}\mathbb{P}\left(\mathrm{deg}(x,\beta_{j_m}\cup\{x\})=j_m\right)\;\mathrm{d}x\\
				&\geq s\frac{\alpha_m^2}{2}\frac{(s\kappa_d\varepsilon^d)^{j_m}}{j_m!}e^{-s\kappa_d3^d\varepsilon^d}p_m\frac{\lambda_d(W)}{2}=:c_{\alpha,k,W,d}s.
			\end{align*}
			\textit{Case 2: $j_m=k$.}\\
			If it exists, we denote by $\ell\in\{1,\dots,n\}$ the index with $j_\ell=k+1$. Then,\begin{align*}
				\hat{\alpha}=\begin{cases}
					\alpha_{\ell}, &\text{ if $\ell$ exists,}\\
					0, &\text{ if $\ell$ does not exist.}
				\end{cases}
			\end{align*}
			Let $\varepsilon>0$ and let $x\in W$ be such that $B^d(x,8\varepsilon)\subset W$. We consider four different configurations to deal with all possible vectors $\alpha=(\alpha_1,\dots,\alpha_n)\in\mathbb{R}^n\backslash\{0\}$ (see Figure \ref{fig:kNN}). Let $e_i$ denote the $d$-dimensional standard unit vector in the $i$-th direction.
			\begin{enumerate}
				\item 	\textit{$k\in\mathbb{N}$ and $\alpha_m(1-k)+\hat{\alpha} k\neq 0$}:\\ 
				In this case we consider the event $S_1$ where for $\hat{x}=x+\frac{3\varepsilon}{4} e_1$ we have $\eta_s(B^d(\hat{x},\varepsilon/4))=k+1$, $\eta_s(B^d(x,3\varepsilon)\backslash B^d(\hat{x},\varepsilon/4))=0$ and $\eta_s(A_{3,\varepsilon}(x,y))\geq k$ for all $y\in\eta_s\backslash B^d(x,3\varepsilon)$. Then it follows
				\begin{align*}
					D_x\sum_{i=1}^n\alpha_iV_{j_i}^k=\alpha_m D_xV_{k}^k+\hat{\alpha}D_xV_{k+1}^k=\alpha_m(1-k)+\hat{\alpha} k\neq 0.
				\end{align*}
				\item \textit{$k\geq 3$ and $\alpha_m(1-k)+\hat{\alpha} k= 0$}:\\
				The condition $\alpha_m(1-k)+\hat{\alpha} k= 0$ implies
				\begin{align*}
					\alpha_m(3-k)+\hat{\alpha}(k-2)=2(\alpha_m-\hat{\alpha})=2\frac{\alpha_m}{k}\neq 0.
				\end{align*}
				We consider the event $S_2$ where  $\eta_s(B^d(\hat{x}_i,\varepsilon/16))=1$ for $i\in\{1,\dots,4\}$ with $\hat{x}_j=x+(-1)^j\frac{3\varepsilon}{4}e_1$ for $j\in\{1,2\}$ and $\hat{x}_j=x+(-1)^j\frac{3\varepsilon}{4}e_2$ for $j\in\{3,4\}$, $\eta_s(B^d(x,\varepsilon/4))=k-3$, $\eta_s(B^d(x,3\varepsilon)\backslash (B^d(x,\varepsilon/4)\cup \bigcup_{i=1}^4B^d(\hat{x}_i,\varepsilon/16)))=0$ and $\eta_s(A_{3,\varepsilon}(x,y))\geq k$ for all $y\in\eta_s\backslash B^d(x,3\varepsilon)$. Then we have
				\begin{align*}
					D_x\sum_{i=1}^n\alpha_iV_{j_i}^k=\alpha_m D_xV_{k}^k+\alpha_{\ell}D_xV_{k+1}^k=\alpha_m(3-k)+\hat{\alpha}(k-2)\neq 0.
				\end{align*}
				\item \textit{$k=2$ and $\alpha_m(1-k)+\hat{\alpha} k= 0$:}\\
				In this case we use the event $S_3$ where $\eta_s(B^d(\hat{x}_i,\varepsilon/16))=1$ for $i\in\{1,2,3\}$ with $\hat{x}_j=x+\frac{7\varepsilon}{16}e_1+(-1)^j\frac{7\varepsilon}{16}e_2$ for $j\in\{1,2\}$ and $\hat{x}_3=x+\frac{7\varepsilon}{8}e_1$. Additionally, we assume $\eta_s(B^d(x,3\varepsilon)\backslash ( \bigcup_{i=1}^3B^d(\hat{x}_i,\varepsilon/16)))=0$ and $\eta_s(A_{3,\varepsilon}(x,y))\geq k$ for all $y\in\eta_s\backslash B^d(x,3\varepsilon)$. Hence,
				\begin{align*}
					D_x\sum_{i=1}^n\alpha_iV_{j_i}^k=\alpha_m D_xV_{k}^k=\alpha_m\neq 0.
				\end{align*}
				\item \textit{$k=1$  and $\alpha_m(1-k)+\hat{\alpha} k= 0$:}\\
				We look at the event $S_4$ where  
				$\eta_s(B^d(\hat{x}_1,\varepsilon/4))=1$ for $\hat{x}_1=x-\frac{\varepsilon}{4}e_1$, $\eta_s(B^d(\hat{x}_2,\varepsilon/4))=2$ for $\hat{x}_2=x+\frac{3\varepsilon}{4}e_1$, $\eta_s(B^d(x,3\varepsilon)\backslash ( \bigcup_{i=1}^2B^d(\hat{x}_i,\varepsilon/4)))=0$ and $\eta_s(A_{3,\varepsilon}(x,y))\geq k$ for all $y\in\eta_s\backslash B^d(x,3\varepsilon)$. Since $\hat{\alpha}=0$, it follows
				\begin{align*}
					D_x\sum_{i=1}^n\alpha_iV_{j_i}^k=2\alpha_m\neq 0.
				\end{align*}
			\end{enumerate}
		\begin{figure}[h]
			\centering
			\label{key}
				\begin{tikzpicture}
				\draw(-80pt,0pt)--(320pt,0pt);
				\draw(-80pt,80pt)--(320pt,80pt);
				\draw(-80pt,0pt)--(-80pt,80pt);
				\draw(0pt,0pt)--(0pt,80pt);
				\draw(80pt,0pt)--(80pt,80pt);
				\draw(160pt,0pt)--(160pt,80pt);
				\draw(240pt,0pt)--(240pt,80pt);
				\draw(320pt,0pt)--(320pt,80pt);
				\draw (-40pt,-10pt) node{Case 1};
				\draw (40pt,-10pt) node{Case 2.1};
				\draw (120pt,-10pt) node{Case 2.2};
				\draw (200pt,-10pt) node{Case 2.3};
				\draw (280pt,-10pt) node{Case 2.4};
				% Case 1
				\draw (-40pt,40pt) circle (36pt);
				\filldraw[red] (-40pt,40pt) circle (1pt);
				\draw[red] (-34.5pt,40pt) node{$x$};
				\draw (-40pt,40pt) -- (-19.4pt,12.2pt);
				\draw (-22pt,25pt) node{$\varepsilon$};
				\filldraw[green] (-36.7pt,30pt) circle (1pt);
				\filldraw[green] (-67pt,50pt) circle (1pt);
				\filldraw[green] (-62pt,28pt) circle (1pt);
				\filldraw[green] (-20pt,62pt) circle (1pt);
				\filldraw[green] (-45pt,65pt) circle (1pt);
				\filldraw[green] (-53pt,55pt) circle (1pt);
				\filldraw[green] (-13pt,45pt) circle (1pt);
				\filldraw[green] (-45pt,15pt) circle (1pt);
				\draw[->](-70pt,10pt) -- (-60pt,20pt);
				\draw[green] (-71pt,7pt) node{\small $j_m$};
				% Case 2.1
				\draw (40pt,40pt) circle (36pt);
				\draw (40pt,67pt) circle (9pt);
				\filldraw[red] (40pt,40pt) circle (1pt);
				\draw[red] (45.5pt,40pt) node{$x$};
				\draw (40pt,40pt) -- (61.6pt,12.2pt);
				\draw (58pt,25pt) node{$\varepsilon$};
				\filldraw[green] (44pt,72pt) circle (1pt);
				\filldraw[green] (36pt,70pt) circle (1pt);
				\filldraw[green] (41pt,63pt) circle (1pt);
				\draw[->](26pt,52pt) -- (38pt,66pt);
				\draw[green] (26pt,48pt) node{\small $k+1$};
				% Case 2.2
				\draw (120pt,40pt) circle (36pt);
				\draw (120pt,40pt) circle (9pt);
				\draw (120pt,67pt) circle (2.25pt);
				\draw (120pt,13pt) circle (2.25pt);
				\draw (147pt,40pt) circle (2.25pt);
				\draw (93pt,40pt) circle (2.25pt);
				\filldraw[red] (120pt,40pt) circle (1pt);
				\draw[red] (125.5pt,40pt) node{$x$};
				\draw (120pt,40pt) -- (141.6pt,12.2pt);
				\draw (138pt,25pt) node{$\varepsilon$};
				\filldraw[green] (120pt,67pt) circle (1pt);
				\filldraw[green] (120pt,13pt) circle (1pt);
				\filldraw[green] (93pt,40pt) circle (1pt);
				\filldraw[green] (147pt,40pt) circle (1pt);
				\filldraw[green] (119pt,43pt) circle (1pt);
				\filldraw[green] (116pt,39pt) circle (1pt);
				\filldraw[green] (120pt,36pt) circle (1pt);
				\draw[->](104pt,53pt) -- (116pt,41pt);
				\draw[green] (104pt,55pt) node{\small $k-3$};
				% Case 2.3
				\draw (200pt,40pt) circle (36pt);
				\draw (215.75pt,55.75pt) circle (2.25pt);
				\draw (200pt,71.5pt) circle (2.25pt);
				\draw (184.25pt,55.75pt) circle (2.25pt);
				\filldraw[red] (200pt,40pt) circle (1pt);
				\draw[red] (205.5pt,40pt) node{$x$};
				\draw (200pt,40pt) -- (221.6pt,12.2pt);
				\draw (218pt,25pt) node{$\varepsilon$};
				\filldraw[green] (215.75pt,55.75pt) circle (1pt);
				\filldraw[green] (200pt,71.5pt) circle (1pt);
				\filldraw[green] (184.25pt,55.75pt) circle (1pt);
				% Case 2.4
				\draw (280pt,40pt) circle (36pt);
				\draw (280pt,31pt) circle (9pt);
				\draw (280pt,67pt) circle (9pt);
				\filldraw[red] (280pt,40pt) circle (1pt);
				\draw[red] (285.5pt,40pt) node{$x$};
				\draw (280pt,40pt) -- (301.6pt,12.2pt);
				\draw (298pt,25pt) node{$\varepsilon$};
				\filldraw[green] (278pt,71pt) circle (1pt);
				\filldraw[green] (280pt,31pt) circle (1pt);
				\filldraw[green] (283pt,64pt) circle (1pt);
			\end{tikzpicture}
		\caption{Configurations in $B^d(x,\varepsilon)$}\label{fig:kNN}
		\end{figure}
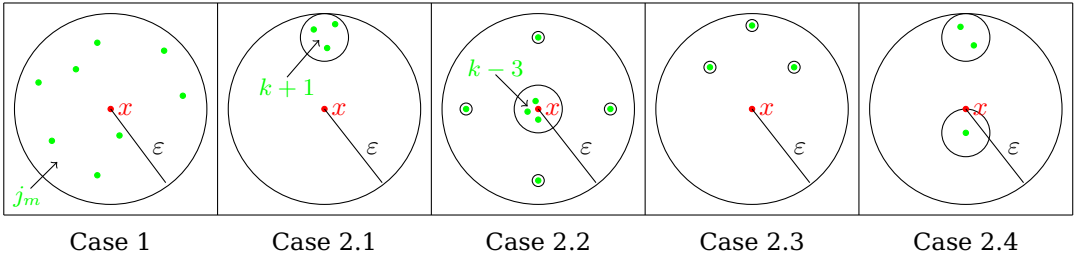
			Let $\varepsilon=\bar{c} s^{-1/d}$ for $\bar{c}>1$ such that $c_1e^{sc_2\varepsilon^d}\leq\frac{1}{2}$. Then, analogously to Case 1, we get $\mathbb{P}(S_u)\geq c_{\alpha,k,d}$ for a constant $c_{\alpha,k,d}>0$ and $u\in\{1,\dots,4\}$. Moreover, let
			\begin{align*}
				c_\alpha=\begin{cases}
					\alpha_m(1-k)+\hat{\alpha} k, &\text{ for }k\in \mathbb{N} \text{ and }\alpha_m(1-k)+\hat{\alpha} k\neq 0,\\
					\alpha_m(3-k)+\hat{\alpha}(k-2), &\text{ for }k\geq 3\text{ and }\alpha_m(1-k)+\hat{\alpha} k=0, \\
					\alpha_m, &\text{ for }k=2 \text{ and }\alpha_m(1-k)+\hat{\alpha} k=0,\\
					2\alpha_{m}, &\text{ for }k=1 \text{ and }\alpha_m(1-k)+\hat{\alpha} k=0.
				\end{cases}
			\end{align*}
			Then, for $A_s=\{x\in W:B^d(x,8\varepsilon)\subset W\}$ and $s$ large enough such that $\lambda_d(A_s)>\frac{\lambda_d(W)}{2}$ it follows for $u\in\{1,\dots,4\}$,
			\begin{align*}
				&\mathbb{E}\left[\int_W \left(D_x\sum_{i=1}^{n}\alpha_iV_{j_i}^{k}\right)^2\;\mathrm{d}\lambda(x)\right]\geq c_\alpha^2\int_{A_s} \mathbb{P}(S_u)\;\mathrm{d}\lambda(x)
				\geq c_{\alpha,k,W,d} s
			\end{align*}
		for a suitable constant $c_{\alpha,k,W,d}>0$.
		
		Our functionals can be written as sums of scores as in \eqref{eq:sum_of_scores}. For $y\in \eta_s$, $j\in\{k,\dots,k_{\mathrm{max}}\}$ and $s\geq 1$ the corresponding score is given by
			\begin{align*}
				\xi_s(y,\eta_s)=\mathbbm{1}\{\mathrm{deg}(y,\eta_s)=j\}.
			\end{align*}
			The scores $(\xi_s)_{s\geq 1}$ clearly fulfil a $(4+p)$-th moment condition and are
			by \cite[proof of Theorem 3.3]{SY21} exponentially stabilising. Therefore, we can apply Lemma \ref{lemma:second_difference_operator}, which completes together with Theorem \ref{thm:varbound} the proof.
		\end{proof}
		
		\begin{rema}\label{rem:inhomogeneous_spatial_random_graphs}
Throughout this section we assume that the underlying Poisson processes have the intensity measures $s\lambda_d|_W$ for $s\geq1$. However, we can generalise our results from these homogeneous Poisson processes to a large class of inhomogeneous Poisson processes. Let $\mu$ be a measure with a density $g: W\to[0,\infty)$ such that $\underline{c}\leq g(x)\leq\overline{c}$ for all $x\in W$ and constants $\underline{c},\overline{c}>0$. All results of this section continue to hold for Poisson processes with intensity measures $s\mu$ for $s\geq 1$. We only have to slightly modify the proofs by bounding the intensity measure by $s \underline{c}\lambda_d|_W$ from below or by $s \overline{c}\lambda_d|_W$ from above depending on whether a lower or an upper bound is required in our estimates. Consequently, some of the constants might change.
\end{rema}
		
		\section{Random Polytopes}\label{sec:random_polytopes}
		The study of the convex hull of random points started with the works \cite{RS63} and \cite{RS64}. In \cite{R05} central limit theorems for the volume and number of $k$-faces as well as variance bounds were shown. Variance asymptotics and central limit theorems for all intrinsic volumes of the convex hull in a ball were derived in \cite{CSY13}. In \cite{LSY19} the rates of convergence for the central limit theorems were further improved.
		
		The $L^p$ surface area measure for a convex body was introduced in \cite{L93}, where the $L^p$ Minkowski problem was described. The Minkowski problem asks for conditions for a Borel measure on the sphere under which this measure is the $L^p$ surface area of a convex body. The discrete $L^p$ Minkowski problem is obtained in the special case, where this convex body is a polytope. This situation can, for example, be found in \cite{HLYZ05} and the references therein. In \cite{HLRT22} the expected $L^p$ surface area of random polytopes was considered as a special case of $T$-functionals of random polytopes.
		
		In this section the two-dimensional vector of $L^p$ surface areas of a random polytope for different $p_1,p_2\in[0,1]$ is considered and lower variance bounds for linear combinations as well as a result on the multivariate normal approximation are derived.
		For $s\geq 1$ let $\eta_s$ be a homogeneous Poisson process on $B^d(0,1)$ with intensity $s$, i.e.\ a Poisson process on $\mathbb{R}^d$ with intensity measure $\lambda=s\lambda_d|_{B^d(0,1)}$, where $\lambda_d|_{B^d(0,1)}$ denotes the restriction of the Lebesgue measure to $B^d(0,1)$. 	
		We consider the random polytope $Q$ generated by $\eta_s\cup\{0\}$, i.e. $Q$ is the convex hull $\mathrm{Conv}(\eta_s\cup\{0\})$. For $p\in[0,1]$ its $L^p$ surface area is given by 
		\begin{align}\label{def l^p surface area}
			A_p = A_p(Q)=\sum_{\text{$F$ facet of $Q$}} \mathrm{dist}(0,F)^{1-p}\lambda_{d-1}(F),
		\end{align}
		where $\mathrm{dist}(0,F)$ stands for the distance of $F$ to the origin $0$ (see for instance \cite[Section 1]{HLRT22}). 
		\begin{theorem}
			\label{theorem:randompoly_var}
			The asymptotic covariance matrix of the vector $s^{(d+3)/(2(d+1))}(A_{p_1},A_{p_2})$ for $p_1,p_2\in [0,1]$ with $p_1\neq p_2$ is positive definite, i.e.\ for any $\alpha=(\alpha_1,\alpha_2)\in\mathbb{R}^2\backslash\{0\}$ there exists a constant $c>0$ such that for $s$ sufficiently large
			\begin{align*}
				\mathrm{Var}[\alpha_1A_{p_1}+\alpha_2A_{p_2}]\geq cs^{-(d+3)/(d+1)}.
			\end{align*}
		\end{theorem}

		Note that we add the origin as an extra point to the Poisson process mainly for technical reasons to ensure a useful definition of the $L^p$ surface area. However, since we are in this section only interested in asymptotic statements for $s\to\infty$, this does not make a difference. Let $\widetilde{Q}$ denote the random polytope that is generated by $\eta_s$, i.e.\ $\widetilde{Q}=\mathrm{Conv}(\eta_s)$, and let $A_p(\widetilde{Q})$ be defined by the right-hand side of \eqref{def l^p surface area}, which is also well-defined if the origin does not belong to the polytope. Since one can choose $m$ disjoint sets $U_1,\dots, U_m\subset B^d(0,1)$ for some $m\in\mathbb{N}$ with $\lambda_d(U_i)>0$, $i\in\{1,\dots,m\}$, such that $0\in\mathrm{Conv}(\xi)$ for all $\xi\in\mathbf{N}$ with $\xi\cap U_i\neq\emptyset$ for all $i\in\{1,\dots,m\}$, we have
		\begin{align}
			\p(A_p(Q)\neq A_p(\widetilde{Q})) &\leq\p(0\notin\mathrm{Conv}(\eta_s))\leq1- \p(\eta_s(U_i) \geq 1 \text{ for } i=1,\dots,m) \nonumber\\
			&=1-\prod_{i=1}^m(1-e^{-s\lambda_d(U_i)}) \leq c_{1,q}e^{-c_{2,q}s} \label{eqn:probability_approximation}
		\end{align}
		for $s\geq1$ with suitable constants $c_{1,q},c_{2,q}>0$.
		 Therefore, the triangle inequality and the estimate $\lvert A_p(Q)-A_p(\widetilde{Q})\rvert\leq 2\kappa_{d}$ provide
		 \begin{align}
		 	\lvert\mathrm{Var}[A_p(Q)]^{1/2}-\Var[A_p(\widetilde{Q})]^{1/2}\rvert^2&\leq \Var[A_p(Q)-A_p(\widetilde{Q})]\leq 	\E[(A_p(Q)-A_p(\widetilde{Q}))^2] \nonumber\\&
		 	\leq (2\kappa_{d})^2c_{1,q}e^{-c_{2,q}s}. \label{eqn:variance_approximation}
		 \end{align}
		and similarly
		\begin{equation}\label{eqn:expectation_approximation}
		\lvert\mathbb{E}[A_p(Q)]-\mathbb{E}[A_p(\widetilde{Q})]\rvert \leq 2\kappa_{d} c_{1,q}e^{-c_{2,q}s}.
		\end{equation}
		Thus, we consider $A_p(\widetilde{Q})$ instead of $A_p(Q)$ throughout this section and, especially, in the proof of Theorem \ref{theorem:randompoly_var}.
	
		We work in the general framework described in Appendix \ref{appendix:stabilising_functionals} with the underlying space $\mathbb{X}=B^d(0,1)$ and the metric
		\begin{align*}
			d_{\max}(x,y)=\max{\{\lVert x-y\rVert,\sqrt{\lvert \lVert x\rVert-\lVert y\rVert \lvert}\}}
		\end{align*}
		for $x,y\in B^d(0,1)$. 
		To prove condition \eqref{condition}, we start with writing the difference of the surface area of the ball $B^d(0,1)$ and the $L^p$ surface area of the random polytope $\widetilde{Q}$ as a sum of scores. The following arguments are mostly analogously to \cite[Section 3.4]{LSY19}, where similar representations for intrinsic volumes were derived. Especially, because the surface area is twice the $(d-1)$-st intrinsic volume, it was shown in \cite[Lemma 3.8]{LSY19} that
		\begin{align*}
			s(\lambda_{d-1}(\partial B^d(0,1))-\lambda_{d-1}(\partial \widetilde{Q}))= 2 \sum_{x\in\eta_s}\xi_{d-1,s}(x,\eta_{s})
		\end{align*}
		with the scores $\xi_{d-1,s}$ as in \cite[last display on p.\ 960]{LSY19} for $s\geq 1$ and where $\partial A$ denotes the boundary of a set $A\subseteq B^d(0,1)$.
		We consider analogous scores $\xi_s$ for the $L^p$ surface area, i.e.\
		\begin{align*}
			\xi_s(x,\eta_s)=2\xi_{d-1,s}(x,\eta_s)+\frac{s}{d}\sum_{F\in\mathcal{F}:x\in F}(1-\mathrm{dist}(0,F)^{1-p})\lambda_{d-1}(F)
		\end{align*}
		for $x\in\eta_s$, where $\mathcal{F}$ denotes the set of all facets of $\widetilde{Q}$. Therefore, we have
		\begin{align*}
			&\sum_{x\in\eta_s}\xi_s(x,\eta_s)
			=\sum_{x\in\eta_s}\Big( 2 \xi_{d-1,s}(x,\eta_s)+\frac{s}{d}\sum_{F\in\mathcal{F}:x\in F}(1-\mathrm{dist}(0,F)^{1-p})\lambda_{d-1}(F)\Big)\\
			&=2\sum_{x\in\eta_s}\xi_{d-1,s}(x,\eta_s)+\sum_{x\in\eta_s}\frac{s}{d}\sum_{F\in\mathcal{F}:x\in F}(1-\mathrm{dist}(0,F)^{1-p})\lambda_{d-1}(F)\\&=s\lambda_{d-1}(\partial B^d(0,1))-s\lambda_{d-1}(\partial \widetilde{Q})+s\lambda_{d-1}(\partial\widetilde{Q})-sA_p(\widetilde{Q})\\
			&=s(\lambda_{d-1}(\partial B^d(0,1))-A_p(\widetilde{Q})).
		\end{align*}
				Fix $\rho_0\in(0,\frac{1}{4})$ and let $B_{-\rho_0}=B^d(0,1)\backslash B^d(0,1-\rho_0)$. In the following Lemma \ref{lemma_scores} and the proof of Theorem \ref{theorem:randompoly_var} we consider slightly modified scores, which are defined by
		\begin{align}
			\label{eq:modified_scores}
			\tilde{\xi}_s(x,\eta_{s})=\mathbbm{1}\{x\in B_{-\rho_0}\}\xi_s(x,(\eta_s\cap B_{-\rho_0})\cup\{0\})
		\end{align}
		for $x\in\eta_s$, $s\geq 1$, and
		\begin{align*}
			\widetilde{A}_p=\sum_{x\in\eta_s}\tilde{\xi}_s(x,\eta_{s}).
		\end{align*}
		We establish that the scores $\tilde{\xi}_s$ have some crucial properties. For exact definitions we refer to Appendix \ref{appendix:stabilising_functionals}.
		\begin{lemma}
			\label{lemma_scores}
			The scores $\tilde{\xi}_s$ are exponentially stabilising with $\alpha_{stab}=d+1$, decay exponentially fast with the distance to the boundary $\partial B^d(0,1)$ with $\alpha_K=d+1$ and fulfil a $q$-th moment condition for $q\geq 1$.
		\end{lemma}
		\begin{proof}		
			Analogously to \cite[Lemma 3.10, Lemma 3.11 and Lemma 3.12]{LSY19} one can show that the scores are exponentially stabilising and decay exponentially fast with the distance to the boundary $\partial B^d(0,1)$.
			
			Let $R(x,\eta_s\cup\{x\})$ denote the corresponding radius of stabilisation with respect to the $d_\mathrm{max}$- distance that is derived in \cite[p.\ 963]{LSY19} and
			let $\tilde{\xi}_{d-1,s}$ denote the slightly adjusted version of the score $\xi_{d-1,s}$, which is defined as $\xi_s$ in \eqref{eq:modified_scores}.
			
			In order to show a $q$-th moment condition for $p\in[0,1]$ we use that
			\begin{align*}
				\bigcup_{F\in\mathcal{F}:x\in F}F\subseteq B_{\mathrm{max}}^d\left(x,R(x,\eta_s\cup\{x\})\right)\subseteq B^d\left(x,R(x,\eta_s\cup\{x\})\right),
			\end{align*}
			where $B_{\mathrm{max}}^d$ denotes the ball with respect to the $d_\mathrm{max}$-distance. Recall that $\mathcal{F}$ stands for the set of all facets of the random polytope. Hence, due to monotonicity of the surface area of convex sets we have
			\begin{align}
				\label{surface_area_est}
				\sum_{F\in\mathcal{F}:x\in F}\lambda_{d-1}(F)\leq d\kappa_d R(x,\eta_s\cup\{x\})^{d-1}.
			\end{align}
			Let $\widetilde{H}$ be the hyperplane through $\partial B^d\big(x,R(x,\eta_s\cup\{x\})\big)\cap \partial B^d(0,1)$. By the definition of the radius of stabilisation in \cite[p.\ 963]{LSY19}, we know that for each vertex $x$ of the random polytope with $R(x,\eta_s\cup\{x\})\leq 1$, $[0,x]$ intersects $\widetilde{H}$, where $[0,x]$ denotes the line connecting $0$ and $x$. Moreover, we get with \cite[p.\ 963]{LSY19} that for a vertex $x$ the distance of the origin to a facet that contains $x$ is at least as large as the distance from the origin to the hyperplane $\widetilde{H}$. Hence, for a facet $F$ that contains $x$ we have
			\begin{align}
				\label{radius_est}
				\mathrm{dist}(0,F)\geq \mathrm{dist}(0,\widetilde{H})\geq\sqrt{1-R(x,\eta_s\cup\{x\})^2}\geq 1-R(x,\eta_s\cup\{x\})^2
			\end{align}
			since the radius of the $(d-1)$-dimensional ball $\widetilde{H}\cap B^d(0,1)$ can be bounded from above by $R(x,\eta_s\cup\{x\})$. The bound in \eqref{radius_est} is obviously also true for $R(x,\eta_s\cup\{x\})>1$.
			
			Since $\mathrm{dist}(0,F)\leq 1$, it holds that $\mathrm{dist}(0,F)^{1-p}\geq \mathrm{dist}(0,F)$ for $p\in[0,1]$ and thus with \eqref{surface_area_est} and \eqref{radius_est} we have for $x\in B_{-\rho_0}$,
			\begin{align*}
				\lvert 	\tilde{\xi}_s(x,\eta_s)\rvert&=\Big\lvert2\tilde{\xi}_{d-1,s}(x,\eta_s)+\frac{s}{d}\sum_{F\in\mathcal{F}:x\in F}(1-\mathrm{dist}(0,F)^{1-p})\lambda_{d-1}(F)\Big\rvert \\
				&\leq 2 \lvert\tilde{\xi}_{d-1,s}(x,\eta_s)\rvert+\frac{s}{d}\sum_{F\in\mathcal{F}:x\in F}\lvert (1-\mathrm{dist}(0,F))\rvert \lambda_{d-1}(F)\\
				&\leq 2 \lvert\tilde{\xi}_{d-1,s}(x,\eta_s)\rvert+\frac{s}{d}R(x,\eta_s\cup\{x\})^2\sum_{F\in\mathcal{F}:x\in F} \lambda_{d-1}(F)\\
				&\leq 2 \lvert\tilde{\xi}_{d-1,s}(x,\eta_s)\rvert+\kappa_dsR(x,\eta_s\cup\{x\})^{d+1}.
			\end{align*}
			Combining this with the fact from \cite[Lemma 3.11]{LSY19} that there exist constants $C_{stab},c_{stab}>0$ such that
			$$
			\p(R(x,\eta_s\cup\{x\})\geq r)\leq C_{stab} \exp[-c_{stab} s r^{d+1}]
			$$
			for $x\in B^d(0,1)$, $r\geq0$, $s\geq 1$ and \cite[Lemma 3.13]{LSY19} that says that the scores $\tilde{\xi}_{d-1,s}$ fulfil a $q$-th moment condition provides the $q$-th moment condition for $\tilde{\xi}_s$.
		\end{proof}
		
Combining Lemma \ref{lemma_scores} with the arguments from the proof of \cite[Lemma 3.9]{LSY19}, we derive that there exist constants $\bar{C}_p,\bar{c}_p>0$ such that
$$
\max\big\{ \p(s A_p(\widetilde{Q})\neq \widetilde{A}_p), |\mathbb{E}[s A_p(\widetilde{Q})] - \mathbb{E}[\widetilde{A}_p], |\mathrm{Var}[s A_p(\widetilde{Q})] - \mathrm{Var}[\widetilde{A}_p]| \big\} \leq \bar{C}_p \exp[-\bar{c}_p s]
$$
for $s\geq 1$. Together with \eqref{eqn:probability_approximation}, \eqref{eqn:variance_approximation} and \eqref{eqn:expectation_approximation} we obtain
\begin{equation}\label{eqn:approximation_A_p}
\max\big\{\p(s A_p\neq \widetilde{A}_p), |\mathbb{E}[s A_p] - \mathbb{E}[\widetilde{A}_p]|, |\mathrm{Var}[s A_p] - \mathrm{Var}[\widetilde{A}_p]| \big\} \leq \hat{C}_p \exp[-\hat{c}_p s]
\end{equation}
for $s\geq 1$ with constants $\hat{C}_p,\hat{c}_p>0$.
		
		Let $S(y^{(1)},\dots,y^{(m)})$ denote the simplex with vertices $y^{(1)},\dots,y^{(m)}$  for $m\in\{1,\dots,d+1\}$.
		For the proof of Theorem \ref{theorem:randompoly_var} we need to know how the $L^p$ surface area of a polytope changes if we add a simplex on one of its facets. Let this $d$-dimensional simplex be given by $S(z^{(1)},\dots,z^{(d+1)})$ for points $z^{(1)},\dots,z^{(d+1)}\in B^d(0,1)$, where $z^{(d+1)}$ denotes the point that is added and $S(z^{(1)},\dots,z^{(d)})$ is the original facet of the polytope. The facets of the simplex are given by $F_i=S(z^{(1)},\dots,z^{(i-1)},z^{(i+1)},\dots,z^{(d+1)})$ and the distance of a facet to the origin is denoted by $\rho_i=\mathrm{dist}(F_i,0)$ for $i\in\{1,\dots,d+1\}$. We are interested in 
		\begin{align}\label{def_Delta}
			\Delta_{p}=\sum_{i=1}^{d}\rho_i^{1-p}\lambda_{d-1}(F_i)-\rho_{d+1}^{1-p}\lambda_{d-1}(F_{d+1}),
		\end{align}
		which is the change of the $L^p$ surface area after adding the simplex.
		
		In the following we also use the notation $\bar{h}=\mathrm{dist}(z^{(d+1)},F_{d+1})$ for the height of the added simplex, $T_i=S(z^{(1)},\dots,z^{(i-1)},z^{(i+1)},\dots,z^{(d)})$ for the $(d-2)$-dimensional faces of the base of the simplex and $h_i=\mathrm{dist}(\bar{z}_{d+1},T_i)$ for $i\in\{1,\dots,d\}$, where $\bar{z}_{d+1}$ is the projection of $z^{(d+1)}$ to $F_{d+1}$.
		The behaviour of $\Delta_{p}$ is described in the following geometric lemma.
		\begin{lemma}\label{lemma_geometry}
			Let $z^{(1)},\dots,z^{(d+1)}\in B^d(0,1)$.
		For a simplex  $S(z^{(1)},\dots,z^{(d+1)})$, whose vertices are chosen in such a way that $\arg\min_{i=1,\dots,d+1}\rho_i=d+1$ we have
			\begin{align*}
				\Big\lvert \Delta_{p}-\frac{1}{d-1}\sum_{i=1}^d\lambda_{d-2}(T_i)\Big(\sqrt{h_i^2+\bar{h}^2}-h_i\Big)\Big\rvert\leq\rho_{d+1}^{-p}(1-\rho_{d+1})\sum_{i=1}^{d+1}\lambda_{d-1}(F_i)
			\end{align*}
		for $p\in[0,1]$ and
		\begin{align*}
			\Big\lvert &\Delta_{p_1}-\Delta_{p_2}-\sum_{i=1}^{d}(p_2-p_1)(\rho_i-\rho_{d+1})\lambda_{d-1}(F_i)\Big\rvert\\&\leq 2 \rho_{d+1}^{-p_2-1} (1-\rho_{d+1})^2\sum_{i=1}^d\lambda_{d-1}(F_i)+\rho_{d+1}^{-p_2}(1-\rho_{d+1})\sum_{i=1}^d\lambda_{d-2}(T_i)\Big(\sqrt{h_i^2+\bar{h}^2}-h_i\Big)
		\end{align*}
		for $p_1,p_2\in[0,1]$ with $p_1<p_2$.
		\end{lemma}
		\begin{proof}
			For $i\in\{1,\dots,d+1\}$ let $F_{d+1}^{(i)}=S(z^{(1)},\dots,z^{(i-1)},z^{(i+1)},\dots,z^{(d)},\bar{z}_{d+1})$. Then, we have
			\begin{align*}
				\lambda_{d-1}(F_{d+1})=\sum_{i=1}^d\lambda_{d-1}(F_{d+1}^{(i)})=\frac{1}{d-1}\sum_{i=1}^{d}\lambda_{d-2}(T_i)h_i
			\end{align*}
		and
		\begin{align*}
			\sum_{i=1}^{d}\lambda_{d-1}(F_i)=\sum_{i=1}^d\frac{1}{d-1}\lambda_{d-2}(T_i)\mathrm{dist}(z^{(d+1)},T_i)=\frac{1}{d-1}\sum_{i=1}^d\lambda_{d-2}(T_i)\sqrt{h_i^2+\bar{h}^2}.
		\end{align*}
		Hence,
		\begin{align}
			\label{eq_difference_fi}
			\sum_{i=1}^{d}\lambda_{d-1}(F_i)-\lambda_{d-1}(F_{d+1})=\frac{1}{d-1}\sum_{i=1}^d\lambda_{d-2}(T_i)\Big(\sqrt{h_i^2+\bar{h}^2}-h_i\Big).
		\end{align}
		Note that, due to the mean value theorem and the assumption $\arg\min_{i=1,\dots,d+1}\rho_i=d+1$, one has for $i\in\{1,\dots,d+1\}$,
		\begin{align*}
			0\leq 1-\rho_i^{1-p}\leq (1-p)\rho_i^{-p}(1-\rho_i)\leq (1-p)\rho_{d+1}^{-p}(1-\rho_{d+1}).
		\end{align*}
		Thus,
		\begin{align*}
			0\leq \sum_{i=1}^{d+1}\lambda_{d-1}(F_i)(1-\rho_i^{1-p})\leq \sum_{i=1}^{d+1}\lambda_{d-1}(F_i)(1-p)\rho_{d+1}^{-p}(1-\rho_{d+1}).
		\end{align*}
	Therefore, it follows with \eqref{def_Delta} and \eqref{eq_difference_fi} that
	\allowdisplaybreaks
	\begin{align*}
			&\Big\lvert \Delta_{p}-\frac{1}{d-1}\sum_{i=1}^d\lambda_{d-2}(T_i)\Big(\sqrt{h_i^2+\bar{h}^2}-h_i\Big)\Big\rvert\\&=\Big\lvert \sum_{i=1}^d (\rho_i^{1-p}-1)\lambda_{d-1}(F_i) -(\rho_{d+1}^{1-p}-1)\lambda_{d-1}(F_{d+1})\Big\rvert\\
			&\leq \sum_{i=1}^{d+1}\lambda_{d-1}(F_i)(1-\rho_i^{1-p})\leq \sum_{i=1}^{d+1}\lambda_{d-1}(F_i)(1-p)\rho_{d+1}^{-p}(1-\rho_{d+1})\\&\leq\rho_{d+1}^{-p}(1-\rho_{d+1})\sum_{i=1}^{d+1}\lambda_{d-1}(F_i).
	\end{align*}

	For the second inequality we have for $p_1<p_2$,
	\begin{align*}
		 \Delta_{p_1}-\Delta_{p_2}&=\sum_{i=1}^{d}(\rho_i^{1-p_1}-\rho_i^{1-p_2})\lambda_{d-1}(F_i)-(\rho_{d+1}^{1-p_1}-\rho_{d+1}^{1-p_2})\lambda_{d-1}(F_{d+1})\\
		 &=\sum_{i=1}^{d}(\rho_i^{1-p_1}-\rho_i^{1-p_2}-(\rho_{d+1}^{1-p_1}-\rho_{d+1}^{1-p_2}))\lambda_{d-1}(F_i)\\&\;\;\;\;+(\rho_{d+1}^{1-p_1}-\rho_{d+1}^{1-p_2})\Big(\sum_{i=1}^d\lambda_{d-1}(F_i)-\lambda_{d-1}(F_{d+1})\Big).
	\end{align*}
	The mean value theorem leads to
	\begin{align*}
		\lvert\rho_{d+1}^{1-p_1}-\rho_{d+1}^{1-p_2}\rvert=\rho_{d+1}^{1-p_1}\lvert1-\rho_{d+1}^{p_1-p_2}\rvert & \leq \rho_{d+1}^{1-p_1}(p_2-p_1)\rho_{d+1}^{p_1-p_2-1}(1-\rho_{d+1}) \\
		& = (p_2-p_1)\rho_{d+1}^{-p_2}(1-\rho_{d+1}).
	\end{align*}
	Together with \eqref{eq_difference_fi} it follows that
	\begin{align}
		\Big\lvert &\Delta_{p_1}-\Delta_{p_2}-\sum_{i=1}^{d}(\rho_i^{1-p_1}-\rho_i^{1-p_2}-(\rho_{d+1}^{1-p_1}-\rho_{d+1}^{1-p_2}))\lambda_{d-1}(F_i)\Big\rvert\nonumber\\&\leq (p_2-p_1) \rho_{d+1}^{-p_2} (1-\rho_{d+1}) \frac{1}{d-1}\sum_{i=1}^d\lambda_{d-2}(T_i)\Big(\sqrt{h_i^2+\bar{h}^2}-h_i\Big).\label{proof_lemma_4.3}
	\end{align}
	For $u,v\in[0,1]$ with $u\geq v$ and $\tau\in[0,1]$ Taylor approximation provides
	\begin{align*}
		\lvert u^\tau-v^\tau-\tau(u-v)\rvert\leq\tau(1-\tau)\Big(\frac{u^{\tau-2}}{2}(1-u)^2+\frac{v^{\tau-2}}{2}(1-v)^2\Big)\leq \tau(1-\tau)v^{\tau-2}(1-v)^2.
	\end{align*}
	Applying this inequality for $\tau=1-p_1$ or $\tau=1-p_2$, $u=\rho_i$ and $v=\rho_{d+1}$, we derive together with \eqref{proof_lemma_4.3} and $\rho_{d+1}\leq 1$,
	\allowdisplaybreaks
	\begin{align*}
		\Big\lvert &\Delta_{p_1}-\Delta_{p_2}-\sum_{i=1}^{d}(p_2-p_1)(\rho_i-\rho_{d+1})\lambda_{d-1}(F_i)\Big\rvert\\&\leq\sum_{i=1}^d((1-p_1)p_1\rho_{d+1}^{-p_1-1}+(1-p_2)p_2\rho_{d+1}^{-p_2-1})(1-\rho_{d+1})^2\lambda_{d-1}(F_i)\\&\;\;\;\;+(p_2-p_1)\rho_{d+1}^{-p_2}(1-\rho_{d+1})\frac{1}{d-1}\sum_{i=1}^d\lambda_{d-2}(T_i)\Big(\sqrt{h_i^2+\bar{h}^2}-h_i\Big)\\
		&\leq 2\rho_{d+1}^{-p_2-1}(1-\rho_{d+1})^2\sum_{i=1}^d\lambda_{d-1}(F_i)+\rho_{d+1}^{-p_2}(1-\rho_{d+1})\sum_{i=1}^d\lambda_{d-2}(T_i)\Big(\sqrt{h_i^2+\bar{h}^2}-h_i\Big),
	\end{align*}
	which completes the proof.
		\end{proof}

		In order to derive Theorem \ref{theorem:randompoly_var} from Theorem \ref{thm:varbound}, we consider the situation that adding an additional point increases the random polytope by exactly one simplex over an existing facet. Lemma \ref{lemma_geometry} allows us to control the corresponding change of the $L^p$ surface area. The main challenge of the following proof is to show that the described situation is sufficiently likely.
		\begin{proof}[Proof of Theorem \ref{theorem:randompoly_var}]
			Let $a>0$ be fixed. Throughout the proof we choose $s\geq 1$ depending on $a$ large enough such that several conditions hold.
			Recall that $e_i$ denotes the standard unit vector in the $i$-th direction and define $x^{(d+1)}=(1-as^{-2/(d+1)})e_1$. Let $x^{(1)},\dots, x^{(d)}\in B^d(0,1)$ be points on the hyperplane
			$$
			H=\{y=(y_1,\dots,y_d)\in\mathbb{R}^d:y_1=1-(a+a^2)s^{-2/(d+1)}\}
			$$
			of pairwise distance $2\ell=2\sqrt{a}s^{-1/(d+1)}$ that form a regular $(d-1)$-dimensional simplex $S$ such that all points have the same distance to $x^{(d+1)}$. Then, $x^{(1)},\dots, x^{(d+1)}$ are the vertices of a $d$-dimensional simplex with height $h=a^2s^{-2/(d+1)}$.
			For a set $A\subset B^d(0,1)$ and $x\in B^d(0,1)\backslash\mathrm{int}(A)$ let 
			\begin{align*}
				\mathrm{Vis}(x,A)=\{y\in B^d(0,1):[y,x]\cap \mathrm{int}(A)=\emptyset\}
			\end{align*} denote the visibility region at $x$, where $\mathrm{int}(A)$ stands for the interior of $A$. Recall that $[y,x]$ denotes the line connecting $x$ and $y$. Let $\varepsilon_h,\varepsilon_\ell\in(0,1/4)$, which will be chosen sufficiently small such that some properties are satisfied throughout this proof. Now, choose $d$ cuboids $C_1^x,\dots, C_{d}^x\subset \mathrm{Vis}(x^{(d+1)},\mathrm{Conv}(x^{(1)},\dots,x^{(d+1)}))$ containing $x^{(1)},\dots, x^{(d)}$ each with height $\varepsilon_ha^2s^{-2/(d+1)}$ and such that its $(d-1)$-dimensional base is a cube of side length $\varepsilon_\ell \sqrt{a}s^{-1/(d+1)}$ which is contained in the hyperplane $H$.
			
			Indeed, $\varepsilon_h,\varepsilon_\ell\in(0,1/4)$ can be chosen small enough such that $C_1^x,\dots, C_{d}^x\subset B^d(0,1)$ because by e.g.\ \cite[Section 6, p.\ 367]{B92} the height $h_k$ of a $k$-dimensional regular simplex $S_k$ with edge length $2\ell$ is given by
			\begin{align}\label{h(Sk)}
				h_k(S_k)=\frac{2\ell}{\sqrt{2}}\sqrt{\frac{k+1}{k}},
			\end{align}
			i.e. for $y\in C_i^x$ with $i\in\{1,\dots, d\}$ we have
			\begin{align*}
				\lVert y\rVert^2&\leq (1-(a+a^2-\varepsilon_ha^2)s^{-2/(d+1)})^2+\left(\frac{d-1}{d}h_{d-1}(S_{d-1})+(d-1)\varepsilon_\ell\sqrt{a} s^{-1/(d+1)}\right)^2\\
				&=(1-(a+a^2-\varepsilon_ha^2)s^{-2/(d+1)})^2+\left(\sqrt{\frac{2(d-1)a}{d}}s^{-1/(d+1)}+(d-1)\varepsilon_\ell\sqrt{a} s^{-1/(d+1)}\right)^2\\
				&=1-\Big[2\Big(\frac{a}{d}+a^2-\varepsilon_ha^2\Big)-2(d-1)\sqrt{\frac{2(d-1)}{d}}\varepsilon_\ell a-(d-1)^2\varepsilon_\ell^2a\\&\quad\quad\quad\quad-(a+a^2-\varepsilon_ha^2)^2s^{-2/(d+1)}\Big]s^{-2/(d+1)}<1
			\end{align*}
			for $\varepsilon_h,\varepsilon_\ell\in(0,1/4)$ small enough and $s$ sufficiently large.
			\begin{figure}
				\centering
				
				\begin{tikzpicture}[decoration={brace,amplitude=2mm}]
					centering
					\draw[black] (0,0) -- (0:4) arc (0:180:4) -- cycle;
					\filldraw[black] (0,0) circle (0.06);
					\draw(0,-0.3) node{\footnotesize$0$};
					\filldraw[black] (3,1.5) circle (0.06);
					\draw(3,1.2) node{\footnotesize $x^{(2)}$};
					\filldraw[black] (-3,1.5) circle (0.06);
					\draw(-3,1.2) node{\footnotesize$x^{(1)}$};
					\filldraw[black] (0,3) circle (0.06);
					\draw(0.1,3.4) node{\footnotesize$x^{(d+1)}$};
					\draw (-3,1.5)--(0,3);
					\draw (0,3)--(3,1.5);
					\draw (3,1.5)--(-3,1.5);
					\draw[dashed] (-3,1.5)--(-4,1);
					\draw[dashed] (3,1.5)--(4,1);
					\draw[dashed](-4,1)--(4,1);
					\draw[decorate]  (-4.5,1)--(-4.5,4);
					\draw(-5.75,2.5) node{$c_as^{-2/(d+1)}$};
					\draw[decorate]  (4.5,3)--(4.5,1.5);
					\draw(5.75,2.25) node{$a^2s^{-2/(d+1)}$};
					\draw[decorate]  (4.5,4)--(4.5,3);
					\draw(5.75,3.5) node{$as^{-2/(d+1)}$};
					\draw[decorate]  (0,0.75)--(-3,0.75);
					\draw(-1.25,0.25) node{$\sqrt{a}s^{-1/(d+1)}$};
					\draw[blue](-3.5,1.75)--(-3.5,1.5);
					\draw[blue](-3.5,1.75)--(-3,1.75);
					\draw[blue](-3,1.75)--(-3,1.5);
					\draw[blue](-3.5,1.5)--(-3,1.5);
					\draw[blue, ->](-3.25,2.5)--(-3.25,1.8);
					\draw[blue] (-3.25,2.8) node{$C_1^x$};
					\draw[blue](3.5,1.75)--(3.5,1.5);
					\draw[blue](3.5,1.75)--(3,1.75);
					\draw[blue](3,1.75)--(3,1.5);
					\draw[blue](3.5,1.5)--(3,1.5);
					\draw[blue, ->](3.25,2.5)--(3.25,1.8);
					\draw[blue] (3.25,2.8) node{$C_2^x$};
					
					\filldraw[red] (3.4,1.6) circle (0.06);
					\draw[red](4.1,1.6) node{\footnotesize $z^{(2)}$};
					\filldraw[red] (-3.2,1.7) circle (0.06);
					\draw[red](-4,1.7) node{\footnotesize$z^{(1)}$};
					\filldraw[red] (0,2.75) circle (0.06);
					\draw[red](0.1,2.4) node{\footnotesize$z^{(d+1)}$};
					\draw[red] (-3.2,1.7)--(0,2.75);
					\draw[red] (0,2.75)--(3.4,1.6);
					\draw[red] (3.4,1.6)--(-3.2,1.7);
					\draw[decorate]  (0.5,3)--(0.5,2.75);
					\draw(1.7,2.875) node{$ta^2s^{-2/(d+1)}$};
				\end{tikzpicture}
				\caption{Construction in $B^d(0,1)$ for $d=2$}
				\label{fig:construction}
			\end{figure}
			
			In the sequel, we use the same notation as in the context of Lemma \ref{lemma_geometry}.
			We consider the simplex $S(z^{(1)},\dots,z^{(d+1)})$, where $z^{(i)}\in C_i^x$ for $i\in\{1,\dots,d\}$ and $z^{(d+1)}=x^{(d+1)}-ta^2s^{-2/(d+1)}e_1$ for $t\in[0,1/2]$ (see Figure \ref{fig:construction}). Due to the choice of $C_i^x$ we have for $s$ sufficiently large and $t\in[0,1/2]$,
			\begin{align}\label{rho_d+1_l}
				 \rho_{d+1}\geq 1-(a+a^2)s^{-2/(d+1)}
			\end{align}
			and
			\begin{align}\label{h_quer_l_und_u}
			\frac{a^2}{4} s^{-2/(d+1)}	\leq \Big(a^2-\varepsilon_ha^2-\frac{a^2}{2}\Big)s^{-2/(d+1)}\leq\bar{h}\leq a^2s^{-2/(d+1)},
			\end{align}
			where we used $\varepsilon_h\in(0,1/4)$.
			
			For $i\in\{1,\dots,d\}$ we show in the following that we can control $h_i=\mathrm{dist}(T_i,\bar{z}_{d+1})$ with the choice of $\varepsilon_h,\varepsilon_\ell$ uniformly for $s$ sufficiently large.
			Define $\widetilde{F}_{d+1}=S(x^{(1)},\dots,x^{(d)})$ and let $\bar{z}_{d+1}$ and $\bar{x}_{d+1}$ denote the projections of $z^{(d+1)}$ to $F_{d+1}$ and $\widetilde{F}_{d+1}$, respectively.
			Moreover, let $\widetilde{T}_{i}=S(x^{(1)},\dots,x^{(i-1)},x^{(i+1)},\dots,x^{(d)})$. Then, for each $y\in T_{i}$ there exists a $\tilde{y}\in \widetilde{T}_{i}$ such that $\lVert y-\tilde{y}\rVert\leq (d-1)\varepsilon_\ell\sqrt{a}s^{-1/(d+1)}+\varepsilon_ha^2s^{-2/(d+1)}$. Hence, with \eqref{h(Sk)},
			\begin{align*}
				\sqrt{\frac{2}{d(d-1)}}\sqrt{a}s^{-1/(d+1)} & =\mathrm{dist}(\bar{x}_{d+1},\widetilde{T}_i)\\
				&\leq \mathrm{dist}(\bar{x}_{d+1},T_i)+(d-1)\varepsilon_\ell\sqrt{a}s^{-1/(d+1)}+\varepsilon_ha^2s^{-2/(d+1)}\\
				&\leq \lVert\bar{x}_{d+1}-\bar{z}_{d+1}\rVert+h_i+(d-1)\varepsilon_\ell\sqrt{a}s^{-1/(d+1)}+\varepsilon_ha^2s^{-2/(d+1)}.
			\end{align*}
			For the distance of the projections we have
			\begin{align}
				\label{dist_xz}
				\lVert \bar{x}_{d+1}-\bar{z}_{d+1}\rVert\leq \lVert \bar{x}_{d+1}-z^{(d+1)}\rVert+\lVert z^{(d+1)}-\bar{z}_{d+1}\rVert\leq 2a^2s^{-2/(d+1)}.
			\end{align}
			Hence, we derive for $h_i$,
			\begin{align*}
				h_i&\geq \sqrt{\frac{2}{d(d-1)}}\sqrt{a}s^{-1/(d+1)}-2a^2s^{-2/(d+1)}-(d-1)\varepsilon_\ell\sqrt{a}s^{-1/(d+1)}-\varepsilon_ha^2s^{-2/(d+1)}.
			\end{align*}
			Note that $2a^2 s^{-2/(d+1)}\leq\frac{1}{2} \sqrt{\frac{2}{d(d-1)}}\sqrt{a} s^{-1/(d+1)}$ for $s$ sufficiently large. Therefore, we can choose $\varepsilon_\ell,\varepsilon_h>0$ small enough such that for all $t\in[0,1/2]$ and $s$ sufficiently large,
			\begin{align}\label{h_i_l}
				h_i\geq c_{h,l}\sqrt{a}s^{-1/(d+1)}
			\end{align}
			with a constant $c_{h,l}>0$. Here the constant $c_{h,l}$ does not depend on $a$, while the lower bound for $s$ such that the inequality holds may depend on $a$. The same applies to the inequalities and constants in the sequel if not stated otherwise.	Moreover, using again \eqref{dist_xz} as well as $\varepsilon_h, \varepsilon_\ell \leq 1/4$, we have
			\begin{align}
				h_i&\leq \lVert \bar{x}_{d+1}-\bar{z}_{d+1}\rVert+\mathrm{dist}(\bar{x}_{d+1}, T_i)\nonumber\\&\leq \lVert \bar{x}_{d+1}-\bar{z}_{d+1}\rVert+\mathrm{dist}(\bar{x}_{d+1}, \widetilde{T}_i) +(d-1)\varepsilon_\ell \sqrt{a}s^{-1/(d+1)}+\varepsilon_ha^2s^{-2/(d+1)}\nonumber\\&\leq \frac{9}{4}a^2s^{-2/{(d+1)}}+ \left(\sqrt{\frac{2}{d(d-1)}}+\frac{d-1}{4}\right) \sqrt{a}s^{-1/(d+1)} \nonumber\\ &\leq c_{h,u}\sqrt{a}s^{-1/(d+1)} \label{h_i_u}
			\end{align}
			for a suitable constant $c_{h,u}>0$, $t\in[0,1/2]$, and $s$ sufficiently large.
				
			By e.g. \cite[Section 6, p.\ 367]{B92} the $k$-dimensional volume $\lambda_k$ of a $k$-dimensional regular simplex $S_k$ with edge length $2\ell$ is
			\begin{align}\label{lambda(Sk)}
				\lambda_k(S_k)=\frac{(2\ell)^k}{k!}\sqrt{\frac{k+1}{2^k}}
			\end{align}
			for $k\in\mathbb{N}$.
			By definition $\widetilde{T}_{i}$ with $i\in\{1,\dots,d\}$ is a regular $(d-2)$-dimensional simplex of side length $2\ell=2\sqrt{a}s^{-1/{(d+1)}}$.
			We know that the $(d-2)$-dimensional volume of a $(d-2)$-dimensional regular simplex of side length $2\sqrt{a}$ in $\mathbb{R}^d$ is continuous with regard to translations of the vertices. Therefore, we can choose a cube around each vertex small enough such that moving each vertex within the corresponding cube changes the $(d-2)$-dimensional volume of the $(d-2)$-dimensional simplex only slightly. Due to homogeneity we can transfer this result to a regular simplex of side length $2\sqrt{a}s^{-1/{(d+1)}}$ for all $s\geq 1$, where each side of the cubes is scaled by $s^{-1/{(d+1)}}$. Hence, we can choose $\varepsilon_h,\varepsilon_\ell\in(0,1/4)$ small enough such that with \eqref{lambda(Sk)} for $s$ sufficiently large,
		\begin{align}
			\lambda_{d-2}(T_i)&\geq \frac{1}{2} 	\lambda_{d-2}(S(x^{(1)},\dots,x^{(i-1)},x^{(i+1)},\dots x^{(d)}))=\frac{2^{(d-2)/2}\sqrt{d-1}}{2(d-2)!}(\sqrt{a} s^{-1/(d+1)})^{d-2} \nonumber\\&=: c_{T,l}a^{(d-2)/2}s^{-(d-2)/(d+1)}\label{T_i_l}
		\end{align}
		and
		\begin{align}
			\lambda_{d-2}(T_i)\leq c_{T,u}a^{(d-2)/2}s^{-(d-2)/(d+1)}\label{T_i_u}
		\end{align}
		for a suitable constant $c_{T,u}>0$.
		Together with \eqref{h_i_u}, it holds
		\begin{align*}
			\lambda_{d-1}(F_{d+1})&=\frac{1}{d-1}\sum_{i=1}^d\lambda_{d-2}(T_i)h_i\\&\leq \frac{1}{d-1}\sum_{i=1}^dc_{T,u}a^{(d-2)/2}s^{-(d-2)/(d+1)}c_{h,u}\sqrt{a}s^{-1/(d+1)}
		\end{align*}
		and with \eqref{h_quer_l_und_u},
			\begin{align*}
				\lambda_{d-1}(F_i)&=\frac{1}{d-1}\lambda_{d-2}(T_i)\sqrt{h_i^2+\bar{h}^2}\\
				&\leq \frac{1}{d-1}c_{T,u}a^{(d-2)/2}s^{-(d-2)/(d+1)}\sqrt{c_{h,u}^2as^{-2/(d+1)}+a^4s^{-4/(d+1)}}
			\end{align*}
			for $i\in\{1,\dots,d\}$. Hence, we have for $j\in\{1,\dots,d+1\}$ and $s$ sufficiently large,
			\begin{align}
				\lambda_{d-1}(F_j)\leq c_{F,u}a^{(d-1)/2}s^{-(d-1)/(d+1)}\label{F_i_u}
			\end{align}
			 for a suitable constant $c_{F,u}>0$. Analogously, we have for $s$ sufficiently large,
			\begin{align}
				\label{F_i_l}
				\lambda_{d-1}(F_j)&\geq c_{F,l}a^{(d-1)/2}s^{-(d-1)/(d+1)}
			\end{align}
			for a suitable constant $c_{F,l}>0$ and $j\in\{1,\dots,d+1\}$.
			Due to the fundamental theorem of calculus we have for $x>y>0$,
			\begin{align}\label{FTC_l}
				\sqrt{x^2+y^2}-x=\int_0^{y^2}\frac{1}{2\sqrt{x^2+z}}\;\mathrm{d}z\geq y^2\frac{1}{2\sqrt{x^2+y^2}}\geq \frac{y^2}{2\sqrt{2}x}
			\end{align}
			and 
			\begin{align}\label{FTC_u}
				\sqrt{x^2+y^2}-x\leq \frac{y^2}{2x}.
			\end{align}
			We can assume without loss of generality that $p_1<p_2$ and that $(\alpha_1,\alpha_2)\in\mathbb{R}^2\backslash\{0\}$ satisfies $\alpha_1+\alpha_2\geq 0$. In the following we distinct the cases $\alpha_1\neq-\alpha_2$ and $\alpha_1=-\alpha_2$. For $\alpha_1\neq-\alpha_2$, we have, by Lemma \ref{lemma_geometry},
			\begin{align*}
				&\Big\lvert\alpha_1\Delta_{p_1}+\alpha_2\Delta_{p_2}-(\alpha_1+\alpha_2)\frac{1}{d-1}\sum_{i=1}^d\lambda_{d-2}(T_i)\Big(\sqrt{h_i^2+\bar{h}^2}-h_i\Big)\Big\rvert\\&\leq(\lvert\alpha_1\rvert+\lvert\alpha_2\rvert) \rho_{d+1}^{-p_2} (1-\rho_{d+1})\sum_{i=1}^{d+1}\lambda_{d-1}(F_i).	
			\end{align*}
			Together with \eqref{rho_d+1_l}, \eqref{h_quer_l_und_u},  \eqref{h_i_u}, \eqref{T_i_l}, \eqref{F_i_u} and \eqref{FTC_l} we obtain for $\alpha_1+\alpha_2>0$, $t\in[0,1/2]$ and $s$ sufficiently large,
			\begin{align*}
				&\alpha_1\Delta_{p_1}+\alpha_2\Delta_{p_2}\\&\geq \frac{\alpha_1+\alpha_2}{d-1}\sum_{i=1}^d c_{T,l}a^{(d-2)/2}s^{-(d-2)/(d+1)}\frac{\frac{a^4}{16}s^{-4/(d+1)}}{2\sqrt{2}c_{h,u}a^{1/2}s^{-1/(d+1)}}\\&\;\;\;\;-(\lvert\alpha_1\rvert+\lvert\alpha_2\rvert)2^{p_2}(a+a^2)s^{-2/(d+1)}\sum_{i=1}^{d+1}c_{F,u}a^{(d-1)/2}s^{-(d-1)/(d+1)}\\
				&\geq \tilde{c}_{d}a^{(d+5)/2}s^{-1}-\tilde{c}_{d,p_1,p_2}(a^{(d+3)/2}+a^{(d+1)/2})s^{-1}
			\end{align*}
			for suitable constants $\tilde{c}_{d},\tilde{c}_{d,p_1,p_2}>0$, where we used that $\rho_{d+1}\geq \frac{1}{2}$ for $s$ sufficiently large. 
			Hence, we can fix $a>0$ large enough such that this estimation provides for $\alpha_1\neq-\alpha_2$ the existence of a constant $\tilde{c}_1>0$ for which
			\begin{align}
				\lvert\alpha_1\Delta_{p_1}+\alpha_2\Delta_{p_2}\rvert\geq\tilde{c}_1a^{(d+5)/2}s^{-1}\label{alpha1+alpha2neq0}
			\end{align} 
		for $s$ sufficiently large and $t\in[0,1/2]$.

						\begin{figure}[htb]
				\begin{minipage}[t]{.35\linewidth}
					\begin{tikzpicture}[decoration={brace,amplitude=2mm}, scale=0.8]
						\vspace{100pt}
						\filldraw[black] (0,-2) circle (0.06);
						\draw(0,-2.3) node{\footnotesize$0$};
						\filldraw[black] (2,1.5) circle (0.06);
						\filldraw[black] (-2,1.5) circle (0.06);
						\draw(-2.5,1.5) node{\footnotesize $\widetilde{F}_{d+1}$};
						\filldraw[red](0,3.5) circle (0.06);
						\draw[red](0.1,3.8) node{\footnotesize$z^{(d+1)}$};
						\draw (2,1.5)--(-2,1.5);
						\filldraw[red] (2.2,1.5) circle (0.06);
						\filldraw[red] (-2.2,2.1) circle (0.06);
						\draw[red](-2.8,2) node{\footnotesize$F_{d+1}$};
						\draw[red] (-2.2,2.1)--(2.2,1.5);
						\draw[red] (0,3.5)--(2.2,1.5);
						\draw[red] (-2.2,2.1)--(0,3.5);
						\draw[dashed,blue](0,-2)--(0,3.5);
						\filldraw[blue] (0,1.8)circle (0.06);
						\draw[blue](0.25,1.95) node{\footnotesize$\bar{z}_{0}$};
						\filldraw[blue] (0,1.5)circle (0.06);
						\draw[blue](-0.4,1.25) node{\footnotesize$\bar{x}_{d+1}$};
						\draw[dashed,blue](0,-2)--(0.5087,1.73);
						\filldraw[blue](0.5087,1.73)circle (0.06);
						\draw[blue](1,1.875) node{\footnotesize$u_{d+1}$};
						\draw[dashed,blue](0.5087,1.5)--(0.5087,1.73);
						\filldraw[blue](0.5087,1.5)circle (0.06);
						\draw[blue](0.7,1.3) node{\footnotesize$\bar{u}$};
						\draw[dashed,blue](0,3.5)--(-0.228,1.831);
						\filldraw[blue](-0.228,1.831) circle (0.06);
						\draw[blue](-0.6,2.1) node{\footnotesize$\bar{z}_{d+1}$};
					\end{tikzpicture}
					\caption{Point configuration on $F_{d+1}$ and $\widetilde{F}_{d+1}$}
					\label{fig:point config}
				\end{minipage}
				\begin{minipage}[t]{.65\linewidth}		
					\centering		\raisebox{1.7em}{
						\begin{tikzpicture}[decoration={brace,amplitude=2mm},scale=0.8]
							
							\filldraw[black] (0,0) circle (0.06);
							\draw(0,-0.3) node{\footnotesize$0$};
							
							\filldraw[red](0,5) circle (0.06);
							\draw[red](0.1,5.3) node{\footnotesize$z^{(d+1)}$};
							
							\filldraw[red] (6,2) circle (0.06);
							\draw[red](6.4,2) node{\footnotesize$T_i$};
							\filldraw[red] (-4,3) circle (0.06);
							\draw[red](-4,2.6) node{\footnotesize$z^{(i)}$};
							\draw[red](3.9,3.5) node{\footnotesize$F_i$};
							\draw[red] (-4,3)--(6,2);
							\draw[red] (-4,3)--(0,5);
							\draw[red] (0,5)--(6,2);
							\draw[dashed,blue] (0,5)--(-0.2376,2.624);
							\filldraw[blue] (-0.2376,2.624)circle (0.06);
							\draw[blue](-0.2376,2.3) node{\footnotesize$\bar{z}_{d+1}$};
							\draw[dashed,blue] (0,0)--(0.257,2.57);
							\filldraw[blue] (0.257,2.57)circle (0.06);
							\draw[blue](-0.4,1) node{\footnotesize$\rho_{d+1}$};
							\filldraw[blue] (2,4)circle (0.06);
							\draw[dashed,blue] (0,0)--(2,4);
							\draw[blue](0.8,1) node{\footnotesize$\rho_{i}$};
							\draw[->,thick, green] (1.841,2.4158)--(2,4);
							\filldraw[green] (0.4160, 4.1542)circle (0.06);
							\draw[->,thick,green] (0.257,2.57)--(0.4160, 4.1542);
							\draw[green](2.8,3) node{\footnotesize $\beta_ihu_{d+1}$};
							\draw[->,thick,cyan] (0.4160, 4.1542)--(2,4);
							\draw[cyan](1.2,3.8) node{\footnotesize $v_{i}$};
							\draw[blue](2.3,4.1) node{\footnotesize $u_{i}$};
							\draw[blue](0.7,2.25) node{\footnotesize $u_{d+1}$};
							
					\end{tikzpicture}}
					\caption{Decomposition of the projection of $0$ to $F_i$}
					\label{fig:projection}
				\end{minipage}
			\end{figure}
		
			For $\alpha_1=-\alpha_2$ we fix $a\in(0,1)$. To use the second part of Lemma \ref{lemma_geometry} we need an estimate for
			$\rho_i-\rho_{d+1}$ for $i\in\{1,\dots,d\}$. 
			Let $u_{i}$ be the projection of $0$ to $F_i$ for $i\in\{1,\dots, d+1\}$ and note that $\bar{x}_{d+1}$, which we introduced as the projection of $x^{(d+1)}$ to $\widetilde{F}_{d+1}$, is also the projection of $0$ on $\widetilde{F}_{d+1}$. Then, for every $i\in\{1,\dots, d\}$, there exist a constant $\beta_i\geq 0$ and a vector $v_i$ orthogonal to $u_{d+1}$ such that
			$$
			u_i = (1+\beta_ih) u_{d+1}+v_i
			$$
			and, thus,
			\begin{align*}
				\rho_i^2=\lVert u_i\rVert^2=(1+\beta_ih)^2\lVert u_{d+1}\rVert^2+\lVert v_i\rVert^2=(1+\beta_ih)^2\rho_{d+1}^2+\lVert v_i\rVert^2.
			\end{align*}
			Let $\bar{u}$ be the projection of $u_{d+1}$ to $\widetilde{F}_{d+1}$, while $\bar{z}_0$ is the intersection point of $F_{d+1}$ with the line through $0$ and $z^{(d+1)}$ (see Figure \ref{fig:point config}). We show that we can choose $\varepsilon_h>0$ small enough such that $u_{d+1}$ is very close to $\bar{z}_0$ to ensure a minimum distance from $u_{d+1}$ to $T_i$. It holds 
			\begin{align*}
				\lVert \bar{x}_{d+1}\rVert^2+\lVert \bar{x}_{d+1}-\bar{u}\rVert^2=\lVert \bar{u}\rVert^2\leq\lVert u_{d+1}\rVert^2\leq\lVert \bar{z}_{0}\rVert^2\leq (\lVert \bar{x}_{d+1}\rVert+\varepsilon_ha^2s^{-2/(d+1)})^2,
			\end{align*}
			which implies
			\begin{align*}
				\lVert \bar{x}_{d+1}-\bar{u}\rVert^2\leq 2\lVert \bar{x}_{d+1}\rVert\varepsilon_ha^2s^{-2/(d+1)}+\varepsilon_h^2 a^4s^{-4/(d+1)}.
			\end{align*}
			This provides
			\begin{align*}
				\lVert \bar{z}_0-u_{d+1}\rVert^2&\leq 	\lVert \bar{x}_{d+1}-\bar{u}\rVert^2+ \varepsilon_h^2a^4s^{-4/(d+1)}
				\leq 2\varepsilon_h a^2s^{-2/(d+1)}+ 2 \varepsilon_h^2a^4s^{-4/(d+1)}.
			\end{align*}
			Hence, we can choose $\varepsilon_h\in(0,1/4)$ small enough such that 
			\begin{align}\label{z0-u_d+1}
				\lVert \bar{z}_0-u_{d+1}\rVert\leq \frac{1}{4}\sqrt{\frac{2}{d(d-1)}}as^{-1/(d+1)}=\frac{\sqrt{a}}{4}\sqrt{\frac{2}{d(d-1)}}\ell\leq \frac{1}{4}\sqrt{\frac{2}{d(d-1)}}\ell
			\end{align}
			since $a\in(0,1)$. For $\varepsilon_\ell>0$ small enough such that for $s$ sufficiently large,
				\begin{align*}
					\mathrm{dist} (\bar{z}_0,T_i)\geq \mathrm{dist}(\bar{x}_{d+1},\widetilde{T}_i)-2\varepsilon_ha^2s^{-2/(d+1)}-(d-1)\varepsilon_\ell\sqrt{a}s^{-1/(d+1)}\geq\frac{1}{2} \sqrt{\frac{2}{d(d-1)}}\ell,
				\end{align*}
			\eqref{z0-u_d+1} implies that $\mathrm{dist}(u_{d+1}, T_i)\geq \frac{1}{4}\sqrt{\frac{2}{d(d-1)}}\ell$ for $i\in\{1,\dots,d\}$ and $s$ sufficiently large. Then, for $\lVert v_i\rVert\leq \frac{1}{8} \sqrt{\frac{2}{d(d-1)}}\ell $, $\mathrm{dist}(u_i,T_i)$ is at least $\frac{1}{8}\sqrt{\frac{2}{d(d-1)}}\ell$ since $\mathrm{dist} (u_{d+1},T_i)\leq \|v_i\| + \mathrm{dist}(u_i,T_i)$ (see Figure \ref{fig:projection}). Hence, with the intercept theorem we have together with \eqref{h_quer_l_und_u} and \eqref{h_i_u},
			\begin{align*}
				\rho_i-\rho_{d+1}&\geq\beta_ih\lVert u_{d+1}\rVert = \bar{h} \frac{\mathrm{dist}(u_i,T_i)}{\mathrm{dist}(z^{(d+1)},T_i)} \\
				&\geq \frac{1}{8}\sqrt{\frac{2}{d(d-1)}}\ell\cdot \frac{\bar{h}}{\sqrt{\bar{h}^2+h_i^2}}\\
				&\geq \frac{1}{8}\sqrt{\frac{2a}{d(d-1)}}s^{-1/(d+1)}\cdot \frac{\frac{1}{4}a^2s^{-2/(d+1)}}{\sqrt{a^4s^{-4/(d+1)}+c_{h,u}^2a s^{-2/(d+1)}}}\\
				&\geq c_{\rho,l} a^2s^{-2/(d+1)}
			\end{align*}
			for a suitable constant $c_{\rho,l}>0$.
			If  $\lVert v_i\rVert>\frac{1}{8}\sqrt{\frac{2}{d(d-1)}}\ell$, we have $$\rho_i^2-\rho_{d+1}^2\geq\rho_i^2-(1+\beta_ih)^2\rho_{d+1}^2 =\|v_i\|^2 >\frac{1}{64}\frac{2}{d(d-1)}\ell^2.$$ Hence,
			\begin{align*}
				\rho_i-\rho_{d+1}\geq \frac{2}{64(\rho_i+\rho_{d+1})d(d-1)}\ell^2\geq \frac{1}{64d(d-1)}\ell^2=\frac{a}{64d(d-1)}s^{-2/(d+1)},
			\end{align*}
			i.e. altogether we have
			\begin{align}
				\rho_i-\rho_{d+1}\geq c_{\rho,l,a}s^{-2/(d+1)}\label{rho_i-rho_d+1}
			\end{align}
			for $s$ sufficiently large with a suitable constant $c_{\rho,l,a}>0$ that depends on $a$.
			
			Together with  Lemma \ref{lemma_geometry} and the inequalities \eqref{rho_d+1_l},  \eqref{h_quer_l_und_u}, \eqref{h_i_l}, \eqref{T_i_u}, \eqref{F_i_u}, \eqref{F_i_l}, \eqref{FTC_u},  \eqref{rho_i-rho_d+1} this provides for a fixed $a\in(0,1)$, $t\in[0,1/2]$ and $s$ sufficiently large,
			\begin{align}
				&\Delta_{p_1}-\Delta_{p_2}\geq \sum_{i=1}^d(p_2-p_1) c_{\rho,l,a}s^{-2/(d+1)}c_{F,l}a^{(d-1)/2}s^{-(d-1)/(d+1)}\nonumber\\
				&-2^{p_2+2}(a+a^2)^2s^{-4/(d+1)}\sum_{i=1}^dc_{F,u}a^{(d-1)/2}s^{-(d-1)/(d+1)}\nonumber\\
				&-2^{p_2}(a+a^2)s^{-2/(d+1)}\sum_{i=1}^d c_{T,u}a^{(d-2)/2}s^{-(d-2)/(d+1)}\frac{a^4s^{-4/(d+1)}}{2c_{h,l}a^{1/2}s^{-1/(d+1)}}\nonumber\\
				&=:C_{a,1}s^{-1}-C_{a,2}s^{-(d+3)/(d+1)}-C_{a,3}s^{-(d+3)/(d+1)},\label{alpha1=-alpha2}
			\end{align}
			which can be bounded from below by $\frac{1}{2} C_{a,1}s^{-1}$ for $s$ sufficiently large.
			Altogether, for $\alpha_1\neq -\alpha_2$ we fix $a>0$ sufficiently large such that \eqref{alpha1+alpha2neq0} holds and for $\alpha_1=-\alpha_2$ we fix $a\in(0,1)$ such that \eqref{alpha1=-alpha2} holds. Then, for $$C_\alpha=\begin{cases}
				\frac{1}{2}C_{a,1}, &\text{ for }\alpha_1=-\alpha_2,\\
				\tilde{c}_1a^{(d+5)/2}, &\text{ else},
			\end{cases}$$ it holds that
			\begin{align}
				\lvert\alpha_1\Delta_{p_1}+\alpha_2\Delta_{p_2}\rvert\geq C_\alpha s^{-1}\label{alphadelta1+alphadelta2}
			\end{align}
			for all $t\in[0,1/2]$ and $s$ sufficiently large.
			
			For the application of Theorem \ref{thm:varbound} we consider the situation that $z^{(1)},\dots,z^{(d)}$ are points of the Poisson process and the point $z^{(d+1)}$ is added. To ensure that the change of $\alpha_1\widetilde{A}_{p_1}+\alpha_2\widetilde{A}_{p_2}$ is given by $s(\alpha_1\Delta_{p_1}+\alpha_2\Delta_{p_2})$ we require that  no further points of $\eta_s$ are present which prevent that $z^{(1)},\dots, z^{(d)}$ is a facet of the random polytope or which could be connected to $z^{(d+1)}$ by edges.
			Therefore, we consider the set
			\begin{align*}
			M_s^x=\{y=(y_1,\dots,y_d)\in B^d(0,1):y_1\geq 1-c_as^{-2/(d+1)}\}.
			\end{align*}
			for some constant $c_a>0$, which might depend on $a$.

			First, we show that the constant $c_a>0$ can be chosen independently from $s$ such that $(B^d(0,1)\backslash 	M_s^x)\cap \mathrm{Vis}(z^{(d+1)},\mathrm{Conv}(z^{(1)},\dots,z^{(d)},x^{(d+1)}))=\emptyset$ for all $z^{(1)}\in C_1^x,\hdots,z^{(d)}\in C_d^x$, i.e.\ that  $c_a>0$ can be chosen in such a way that any line on the boundary of $\mathrm{Conv}(z^{(1)},\dots,z^{(d)},x^{(d+1)})$ through $x^{(d+1)}$ meets the hyperplane $\{y=(y_1,\dots,y_d)\in \mathbb{R}^d:y_1=1-c_as^{-2/(d+1)}\}$ outside the ball $B^d(0,1)$.  Note that this implies that $(B^d(0,1)\backslash 	M_s^x)\cap\mathrm{Vis}(z^{(d+1)},\mathrm{Conv}(z^{(1)},\dots,z^{(d+1)}))=\emptyset$ for all $t\in[0,1/2]$. With \eqref{dist_xz} and \eqref{h_i_l} it holds that 
					\begin{align}
						\mathrm{dist}(\bar{x}_{d+1},T_i)&\geq h_i-\lVert \bar{x}_{d+1}-\bar{z}_{d+1}\rVert\geq c_{h,l}\sqrt{a}s^{-1/(d+1)}-2a^2s^{-2/(d+1)}\nonumber\\&\geq \tilde{c}_{h,l}\sqrt{a}s^{-1/(d+1)}\label{eq:dist(x_d+1,T_i)}
					\end{align}
					for $\tilde{c}_{h,l}=\frac{c_{h,l}}{2}$, $i\in\{1,\dots,d\}$ and $s$ sufficiently large.

					Let $B_C^{d-1}=B^{d}(\bar{x}_{d+1},\tilde{c}_{h,l}\sqrt{a}s^{-1/(d+1)})\cap H$. Then, because of \eqref{eq:dist(x_d+1,T_i)}, $$\mathrm{Vis}(x^{(d+1)},\mathrm{Conv}(z^{(1)},\dots,z^{(d)},x^{(d+1)}))$$ is a subset of the visibility region at $x^{(d+1)}$ of the smallest cone $K$ with apex $x^{(d+1)}$ that contains $B_C^{d-1}$. Hence, if we choose $c_a>0$ such that $(B^d(0,1)\backslash 	M_s^x)\cap\mathrm{Vis}(x^{(d+1)},K)=\emptyset$, then also $(B^d(0,1)\backslash 	M_s^x)\cap \mathrm{Vis}(x^{(d+1)},\mathrm{Conv}(z^{(1)},\dots,z^{(d)},x^{(d+1)}))=\emptyset$.
				Because of symmetry it suffices to ensure that the line through $x^{(d+1)}$ and
				\begin{align*}
					\hat{y}&=(1-(a+a^2)s^{-2/{(d+1)}})e_1+\tilde{c}_{h,l}\sqrt{a}s^{-1/(d+1)}e_2
				\end{align*} meets $H$ outside of $B^d(0,1)$. A point $\hat{x}_\gamma$ on the line through $x^{(d+1)}$ and $\hat{y}$ can be described by
				\begin{align}\label{x_gamma2}
					\hat{x}_\gamma=(1-as^{-2/(d+1)})e_1+\gamma(-a^2s^{-2/(d+1)}e_1+\tilde{c}_{h,l}\sqrt{a}s^{-1/(d+1)}e_2)
				\end{align}
				for $\gamma\in\mathbb{R}$. To determine a possible constant $c_a>0$ we need a $\gamma>1$ such that the point $x_\gamma=(x_{\gamma,1},\dots,x_{\gamma,d})$ fulfils $\lVert x_\gamma\rVert>1$. If $x_{\gamma,1}>1-\frac{1}{2}\sum_{i=2}^dx_{\gamma,i}^2\geq \sqrt{1-\sum_{i=2}^dx_{\gamma,i}^2}$, it holds that $x_\gamma\notin B^d(0,1)$, i.e.\ $\hat{x}_{\gamma}\notin B^d(0,1)$ if
				\begin{align}
					1-(a+\gamma a^2)s^{-2/(d+1)}>1-\frac{\gamma^2}{2}\tilde{c}_{h,l}^2a s^{-2/(d+1)} 
					\quad\quad\Longleftrightarrow\quad\quad \frac{\gamma^2}{2}\tilde{c}_{h,l}^2-\gamma a-1>0.\label{eq:inequality_gamma2}
				\end{align}
				This inequality is fulfilled for $\gamma>1$ large enough independently of $s$. 
				Hence, inserting a possible $\hat{\gamma}>1$, which fulfils \eqref{eq:inequality_gamma2}, in \eqref{x_gamma2} provides that $c_a>0$ can be chosen independently from $s$ as $c_a=a+\hat{\gamma}a^2$. From now on let $s$ be sufficiently large such that  $1-c_as^{-2/(d+1)}\geq \rho_0$.
			
			Due to translation invariance, the same configuration of sets can be constructed for any $x\in B^d(0,1)$ with $\lVert x\rVert=1-(a+ta^2)s^{-2/(d+1)}$ for $t\in[0,1/2]$ by defining $M_s^x,C_1^x,\dots,C_{d}^x$ for each $x$ as the suitable rotated regions.  
			Define
			\begin{align*}
				A=\{x\in B^d(0,1):\lVert x\rVert= 1-(a+ta^2)s^{-2/(d+1)}\text{ and }t\in [0,1/2]\}.
			\end{align*}
			Combining our previous considerations leads to
			\begin{align*}
				\alpha_1D_x\widetilde{A}_{p_1}+\alpha_2D_x\widetilde{A}_{p_2}=s(\alpha_1\Delta_{p_1}+\alpha_2\Delta_{p_2})
			\end{align*}
			if
			\begin{align*}
				\eta_s(C_i^x)=1 \quad\text{ for }\quad i\in\{1,\dots,d\} \quad\text{ and }\quad\eta_s\Big(M_s^x\backslash\bigcup\limits_{i=1}^{d}C_i^x\Big)=0.
			\end{align*}
			for $s$ sufficiently large.
			 Together with \eqref{alphadelta1+alphadelta2} we obtain for $s$ sufficiently large
			\allowdisplaybreaks
			\begin{align}
				&\E\left[\int\lvert \alpha_1D_x\widetilde{A}_{p_1}+\alpha_2D_x\widetilde{A}_{p_2}\rvert ^2\;\mathrm{d}\lambda(x)\right]\geq \E\left[\int_{A}	\lvert \alpha_1D_x\widetilde{A}_{p_1}+\alpha_2D_x\widetilde{A}_{p_2}\rvert ^2\;\mathrm{d}\lambda(x)\right]\nonumber \\&\geq s\int_{A} \mathbb{P}(\lvert \alpha_1D_x\widetilde{A}_{p_1}+\alpha_2D_x\widetilde{A}_{p_2}\rvert\geq C_\alpha ) C_\alpha^2\;\mathrm{d}x\nonumber\\
				&\geq C_\alpha^2 s\int_{A}\mathbb{P}\Big(\eta_s\Big(M_s^x\backslash\bigcup\limits_{i=1}^{d}C_i^x\Big)=0,\eta_s(C_1^x)=1,\dots,\eta_s(C_{d}^x)=1\Big) \;\mathrm{d}x\nonumber\\
				&= C_\alpha^2 s\int_{A}\mathbb{P}\Big(\eta_s\Big(M_s^x\backslash\bigcup\limits_{i=1}^{d}C_i^x\Big)=0\Big)\prod_{i=1}^{d}\mathbb{P}(\eta_s(C_i^x)=1) \;\mathrm{d}x.\label{eq:prob polytopes}
			\end{align}
		Due to the definition of $C_i^x$ we know that $\lambda_d(C_i^x)=\varepsilon_ha^2(\varepsilon_\ell\sqrt{a})^{d-1} s^{-1}$ for $i\in\{1,\dots, d\}$, i.e. the volume of the sets $C_i^x$ is of order $s^{-1}$. 
		
		For $\lambda_d(M_s^x)$ we consider at first the radius $r$ of the $(d-1)$-dimensional ball $B_C=\{y=(y_1,\dots,y_d)\in B^d(0,1):y_1= 1-c_as^{-2/(d+1)}\}$. This radius fulfils $r^2+(1-c_as^{-2/(d+1)})^2=1$. Hence,
		\begin{align*}
			r^2=2c_as^{-2/(d+1)}-c_a^2s^{-4/(d+1)}\leq 2c_as^{-2/(d+1)}
		\end{align*}
		and therefore
		$$
			\lambda_d(M_s^x)\leq \kappa_{d-1}r^{d-1}c_as^{-2/(d+1)}\leq\tilde{c}_as^{-1}
		$$
		for $\tilde{c}_a=\kappa_{d-1}c_a(\sqrt{2c_a})^{d-1}$. Thus,
		$\lambda_d(M_s^x\backslash\bigcup\limits_{i=1}^{d}C_i^x)$ is at most of order $s^{-1}$. Therefore, since the Poisson process has intensity $s$, the order of the whole term in \eqref{eq:prob polytopes} can be bounded from below by a multiple of $s^{-1}\lambda_d(A)$, where 
		\begin{align*}
			\lambda_d(A)= \kappa_d \bigg((1-as^{-2/(d+1)})^d-\Big(1-\Big(a+\frac{a^2}{2}\Big)s^{-2/(d+1)}\Big)^d \bigg) \geq \tilde{c} s^{-2/(d+1)}
		\end{align*}
	for a suitable constant $\tilde{c}>0$ and $s$ sufficiently large. Altogether we have
	\begin{align*}
		\E\left[\int\lvert \alpha_1D_x\widetilde{A}_{p_1}+\alpha_2D_x\widetilde{A}_{p_2}\rvert ^2\;\mathrm{d}\lambda(x)\right]\geq \tilde{C} s^{(d-1)/(d+1)}
	\end{align*}
for some constant $\tilde{C}>0$ and $s$ sufficiently large.

		Next, we check condition \eqref{condition}. Due to Lemma \ref{lemma_scores} we can apply the results in \cite[Lemma 5.5 and Lemma 5.9]{LSY19}, i.e.\ there exists a constant $C>0$ satisfying
			\begin{align}
				\mathbb{E}\lvert D_{x}\widetilde{A}_{p_i}(\eta_s\cup U)\rvert^{5} \leq C
				\label{lemma5.52}
			\end{align}
			for $U\subset B^d(0,1)$ with $\lvert U\rvert\leq 1$ and for any $\beta>0$,
			\begin{align}
				s\int\mathbb{P}(D_{x,y}^2\widetilde{A}_{p_i}\neq 0)^\beta\;\mathrm{d}y\leq C_\beta \exp[-c_\beta sd_{\max}(x,\partial B^d(0,1))^{(d+1)}]
				\label{lemma5.92}
			\end{align}
			for some constants $C_\beta, c_\beta>0$ and $x\in B^d(0,1)$. Note that the statements of \cite[Lemma 5.9]{LSY19} contain typos since the exponent $\alpha$ of ${\rm d}_s(x_1,K)$ is missing in the upper bounds.
			Using \eqref{lemma5.52}, the Hölder inequality and Jensen's inequality provides
			\begin{align*}
				&\mathbb{E}\lvert D_{x,y}^2\widetilde{A}_{p_i}\rvert^2= \mathbb{E}\left[\lvert D_{x,y}^2\widetilde{A}_{p_i}\rvert^2\mathbbm{1}\{D_{x,y}^2\widetilde{A}_{p_i}\neq 0\}\right]\\
				&\leq (\mathbb{E}\lvert D_{x,y}^2\widetilde{A}_{p_i}\rvert^{5})^{2/5}\mathbb{P}(D_{x,y}^2\widetilde{A}_{p_i}\neq 0)^{3/5}\\
				&=	\mathbb{E}\lvert D_{x}\widetilde{A}_{p_i}(\eta_s\cup\{y\})-D_{x}\widetilde{A}_{p_i}(\eta_s)\rvert^{5})^{2/5} \mathbb{P}(D_{x,y}^2\widetilde{A}_{p_i}\neq 0)^{3/5}\\
				&\leq \left(2^{4}\left(\mathbb{E}\lvert D_{x}\widetilde{A}_{p_i}(\eta_s\cup\{y\})\rvert^{5}+\mathbb{E}\lvert D_{x}\widetilde{A}_{p_i}(\eta_s)\rvert^{5}\right)\right)^{2/5}\mathbb{P}(D_{x,y}^2\widetilde{A}_{p_i}\neq 0)^{3/5}\\
				&\leq 4C^{2/5}\mathbb{P}(D_{x,y}^2\widetilde{A}_{p_i}\neq 0)^{3/5}
			\end{align*}
			for $i\in\{1,2\}$. Therefore, using Jensen's inequality and \eqref{lemma5.92}, it follows
			\allowdisplaybreaks
			\begin{align*}
				&\mathbb{E}\left[\int_{B^d(0,1)}\int_{B^d(0,1)} \left(D_{x,y}^2\sum_{i=1}^{2}\alpha_i\widetilde{A}_{p_i}\right)^2\;\mathrm{d}\lambda(x)\;\mathrm{d}\lambda(y)\right]\\&
				\leq 2\sum_{i=1}^{2}\alpha_i^2\int_{B^d(0,1)}\int_{B^d(0,1)} 	\mathbb{E}\lvert D_{x,y}^2\widetilde{A}_{p_i}\rvert^2\;\mathrm{d}\lambda(x)\;\mathrm{d}\lambda(y)\\&
				\leq 2\sum_{i=1}^{2}\alpha_i^2 s\int_{B^d(0,1)} s\int_{B^d(0,1)} 4C^{2/5}\mathbb{P}(D_{x,y}^2\widetilde{A}_{p_i}\neq 0)^{3/5}\;\mathrm{d}x\;\mathrm{d}y\\
				&\leq 8 \sum_{i=1}^{2}\alpha_i^2 C^{2/5} s\int_{B^d(0,1)} C_{3/5} \exp[-c_{3/5} s d_{\max}(x,\partial B^d(0,1))^{(d+1)}]\;\mathrm{d}x\\
				&\leq c_{\alpha}^{(1)} s\int_{B^d(0,1)} \exp[-c_{3/5} s (1-\lVert x\rVert)^{(d+1)/2}]\;\mathrm{d}x\\
				&\leq c_{\alpha}^{(2)} s\int_{0}^1 \exp[-c_{3/5} s (1-r)^{(d+1)/2}]\;\mathrm{d}r = c_{\alpha}^{(2)} s\int_{0}^1 \exp[-c_{3/5} s u^{(d+1)/2}]\;\mathrm{d}u \\
				&\leq c_{\alpha}^{(3)} s \int_{0}^{(c_{3/5}s)^{2/(d+1)}} e^{-t^{(d+1)/2}} s^{-2/(d+1)}\mathrm{dt}
				\leq c_{\alpha}^{(4)} ss^{-2/(d+1)} = c_{\alpha}^{(4)} s^{(d-1)/(d+1)}
			\end{align*}
			for suitable constants $c_{\alpha}^{(i)}>0$ for $i\in\{1,2,3,4\}$ and $s$ sufficiently large. This shows  together with Theorem \ref{thm:varbound} that $\Var[\alpha_1\widetilde{A}_{p_1}+\alpha_2\widetilde{A}_{p_2}]\geq cs^{(d-1)/(d+1)}$ for a suitable constant $c>0$. Now \eqref{eqn:approximation_A_p} yields a lower bound of the same order for $\alpha_1sA_{p_1}+\alpha_2sA_{p_2}$, which completes the proof.
		\end{proof}
		As a consequence of the lower variance bound in Theorem \ref{theorem:randompoly_var}, one can derive bounds for the multivariate normal approximation of two $L^p$ surface areas. Therefore, we define the $d_{convex}$-distance. Let $\mathcal{I}$ be the set of indicators of measurable convex sets in $\mathbb{R}^2$. Then, for the two-dimensional random vectors $Y$ and $Z$ the $d_{convex}$-distance is defined as
		\begin{align*}
			d_{convex}(Y,Z)=\sup_{h\in \mathcal{I}} \lvert\E[h(Y)]-\E[h(Z)]\rvert.
		\end{align*}
		
		\begin{theorem}
			\label{theorem:randompoly_limit}
			Let $(A_{p_1},A_{p_2})$ be the vector of $L^p$ surface areas for $p_1,p_2\in[0,1]$ with $p_1\neq p_2$. Denote by $\Sigma(s)$ the covariance matrix of $s^{(d+3)/(2(d+1))}(A_{p_1},A_{p_2})$. Let $N_{\Sigma(s)}$ be a centred Gaussian random vector with covariance matrix $\Sigma(s)$. Then there exists a constant $c>0$ such that
			\begin{align*}
				d_{convex}(s^{(d+3)/(2(d+1))}(A_{p_1}-\E[A_{p_1}],A_{p_2}-\E[A_{p_2}]), N_{\Sigma(s)})\leq c s^{-(d-1)/(2(d+1))}
				\end{align*}
			for $s\geq 1$.
		\end{theorem}
		
		\begin{proof}
		For $s\geq 1$ we define $\widetilde{Z}_s =  s^{-(d-1)/(2(d+1))}(\widetilde{A}_{p_1},\widetilde{A}_{p_2})$. From \cite[Theorem 4.1 c)]{SY21} with $\tau=(d-1)/(2(d+1))$, whose assumptions are satisfied by Lemma \ref{lemma_scores}, it follows that
		\begin{equation}\label{eqn:approximation_d_convex}
		d_{convex}(\widetilde{Z}_s-\mathbb{E}[\widetilde{Z}_s],N_{\Sigma(s)}) \leq \tilde{c} s^{-(d-1)/(2(d+1))}
		\end{equation}
		for $s\geq 1$ with a constant $\tilde{c}>0$ if we can check that
\begin{itemize}
\item [(i)] for any constant $c_I>0$ there exists a constant $\tilde{c}_I>0$ such that
\begin{align*}
s\int_{B^d(0,1)}\exp[-c_Isd_{\max}(x,\partial B^d(0,1))^{(d+1)}]\;\mathrm{d}x\leq \tilde{c}_I s^{(d-1)/(d+1)}
\end{align*}
for $s\geq 1$,
\item [(ii)] $|(\Sigma(s))_{u,v} - \Cov(\widetilde{Z}_s^{(u)},\widetilde{Z}_s^{(v)})|$ is at most of order $s^{-(d-1)/(2(d+1))}$ for all $u,v\in\{1,2\}$,
\item [(iii)] $\|\Sigma(s)^{-1}\|_{op}$ is uniformly bounded for $s$ sufficiently large, where $\lVert\cdot\rVert_{\text{op}}$ denotes the operator norm.
\end{itemize}
Analogously to the calculation at the end of the proof of Theorem \ref{theorem:randompoly_var} one can show (i), while (ii) follows from \eqref{eqn:approximation_A_p}.

In order to establish (iii), we assume that there is a subsequence $(s_n)_{n\in\mathbb{N}}$ such that $\lVert\Sigma(s_n)^{-1}\rVert_{\text{op}}\to\infty$ and $s_n\to\infty$ as $n\to\infty$. From the Poincar\'e inequality (see \eqref{eqn:Poincare}), \eqref{lemma5.52}, \cite[(5.8) in Lemma 5.10]{LSY19} and (i), one deduces that all variances and, thus, all covariances of the components of $\widetilde{Z}_s$ are uniformly bounded for $s\geq1$. By (ii) the same holds for the entries of $\Sigma(s)$. Thus, there exists a subsequence $(s_{n_k})_{k\in\mathbb{N}}$ and a matrix $\Sigma\in\mathbb{R}^{2\times 2}$ such that $\Sigma(s_{n_k})\to\Sigma$ as $k\to\infty$. From Theorem \ref{theorem:randompoly_var} it follows that $\Sigma$ is positive definite as $\alpha^T\Sigma\alpha=\lim\limits_{k\to\infty}\alpha^T\Sigma(s_{n_k})\alpha>0$ for any $\alpha\in\mathbb{R}^2\backslash\{0\}$. Thus, $\lVert\Sigma^{-1}\rVert_{\text{op}}$ is well-defined and $\Vert\Sigma(s_{n_k})^{-1}\rVert_\text{op}\to\lVert\Sigma^{-1}\rVert_{\text{op}}$ as $k\to\infty$. Since this is a contradiction to the assumption, we have shown that $\lVert\Sigma(s)^{-1}\rVert_{\text{op}}$ is uniformly bounded for $s$ sufficiently large, which is (iii) and completes the proof of \eqref{eqn:approximation_d_convex}.

Moreover, let $Z_s =  s^{(d+3)/(2(d+1))}(A_{p_1},A_{p_2})=s^{-(d-1)/(2(d+1))}(sA_{p_1},sA_{p_2})$. It follows from the triangle inequality that
\begin{align*}
& d_{convex}(Z_s-\mathbb{E}[Z_s],N_{\Sigma(s)})\\ & \leq d_{convex}(Z_s-\mathbb{E}[Z_s],\widetilde{Z}_s-\mathbb{E}[Z_s]) + d_{convex}(\widetilde{Z}_s-\mathbb{E}[Z_s],N_{\Sigma(s)}) \\
& \leq \p(Z_s\neq\widetilde{Z}_s) + d_{convex}(\widetilde{Z}_s-\mathbb{E}[\widetilde{Z}_s],N_{\Sigma(s)}+\mathbb{E}[Z_s]-\mathbb{E}[\widetilde{Z}_s]) \\
& \leq \p(Z_s\neq\widetilde{Z}_s) + d_{convex}(\widetilde{Z}_s-\mathbb{E}[\widetilde{Z}_s],N_{\Sigma(s)}) + d_{convex}(N_{\Sigma(s)},N_{\Sigma(s)}+\mathbb{E}[Z_s]-\mathbb{E}[\widetilde{Z}_s]).
\end{align*}
Since the first term on the right-hand side vanishes exponentially fast by \eqref{eqn:approximation_A_p} and the second one was treated in \eqref{eqn:approximation_d_convex}, it remains to study the third term. We have that
\begin{align*}
& d_{convex}(N_{\Sigma(s)},N_{\Sigma(s)}+\mathbb{E}[Z_s]-\mathbb{E}[\widetilde{Z}_s])\\ & = d_{convex}(N_{I},N_{I}+\Sigma(s)^{-1/2}(\mathbb{E}[Z_s]-\mathbb{E}[\widetilde{Z}_s])) \\
& \leq \sup_{K\subseteq\mathbb{R}^2\text{ convex}} \p({\rm dist}(N_I,\partial K)\leq \|\Sigma(s)^{-1/2}(\mathbb{E}[Z_s]-\mathbb{E}[\widetilde{Z}_s])\| ) \\
& \leq \sup_{K\subseteq\mathbb{R}^2\text{ convex}} \p({\rm dist}(N_I,\partial K)\leq \|\Sigma(s)^{-1}\|^{1/2}_{op}\|\mathbb{E}[Z_s]-\mathbb{E}[\widetilde{Z}_s]\| ),
\end{align*}
where $N_I$ is distributed according to a two-dimensional standard normal distribution. From \cite[Corollary 3.2]{BR10} one obtains that the right-hand side is bounded by a constant times
$$
 \|\Sigma(s)^{-1}\|^{1/2}_{op}\|\mathbb{E}[Z_s]-\mathbb{E}[\widetilde{Z}_s]\|.
$$
Now (iii) from above and \eqref{eqn:approximation_A_p} imply that this expression vanishes exponentially fast for $s\to\infty$, which concludes the proof.
		\end{proof}
		
		Theorem \ref{theorem:randompoly_var} and Theorem \ref{theorem:randompoly_limit} especially provide a lower variance bound and a result on the multivariate normal approximation for the vector of surface area and volume of a random polytope since $A_0=dV_d$ and $A_1=S_{d-1}$, where $V_d$ and $S_{d-1}$ denote the volume and surface area, respectively. 
		
		Lower and upper variance bounds of the same order as in Theorem \ref{theorem:randompoly_var} were already derived for the volume in \cite{R05}. For binomial input, analogous variance bounds for intrinsic volumes were shown in \cite{BFV10}. The case of an underlying Poisson process and, in particular, variance asymptotics for intrinsic volumes were discussed in \cite{CSY13}. We expect that variance asymptotics for the $L^p$ surface area and especially the positivity of the asymptotic variance can be derived using the same method as in \cite{CSY13}. However, the proof in \cite{CSY13} cannot be directly transferred to the linear combination of two $L^p$ surface areas because for a linear combinations with scalars of different sign the monotonicity argument in \cite[p. 100]{CSY13} does not work.

In \cite{Grygierek21} the multivariate normal approximation of the vector of all intrinsic volumes and all numbers of lower-dimensional faces of the convex hull of Poisson points in a smooth convex body is considered. As in Theorem \ref{theorem:randompoly_limit}, one compares with a multivariate normal distribution with the same covariance matrix, but as the so-called $d_3$-distance is studied no information about the regularity of the asymptotic covariance matrix is required. In the same work positive linear combinations of intrinsic volumes were considered since for coefficients with different signs it could not be ensured that the corresponding asymptotic variance is positive. For the special case of volume and surface area and an underlying ball, this problem is resolved by Theorem \ref{theorem:randompoly_var}. In contrast to the findings in \cite{Grygierek21}, Theorem \ref{theorem:randompoly_limit} deals with non-smooth test functions and the obtained bounds are of a better order since a logarithmic factor could be removed. The rates of convergence derived in \cite[Section 3]{LSY19} for the univariate normal approximation of intrinsic volumes in Kolmogorov distance are also of the order $s^{-(d-1)/(2(d+1))}$.

\begin{rema}
The results of this section prevail if we assume that the Poisson processes have underlying intensity measures $s\mu$ for $s\geq 0$, where $\mu$ is a measure with a density $g:B^d(0,1)\to[0,\infty)$ satisfying $\underline{c} \leq g(x) \leq \overline{c}$ for all $x\in B^d(0,1)$ and some constants $\underline{c},\overline{c}>0$ (see also Remark \ref{rem:inhomogeneous_spatial_random_graphs}). Moreover, we expect that it is possible to replace the $d$-dimensional unit ball by a compact convex non-empty subset of $\mathbb{R}^d$ with $C^2$-boundary and positive Gaussian curvature. Since the boundaries of these sets as the boundary of the unit ball are locally between two paraboloids, we believe that similar arguments as in \cite[Subsection 3.4]{LSY19} allow to prove our results for this larger class of underlying bodies. However, we did not pursue this approach in order to not further increase the length and complexity of the proofs in this section.
\end{rema}	
		
		\section{Excursion sets of Poisson shot noise processes}\label{sec:excursion_sets}
		
Excursion sets of random fields are an important topic of probability theory and have many applications, for example in biology or engineering. For an introduction into this topic see for instance \cite{AT07}. The most common underlying random fields are Gaussian random fields, but a further prominent choice are Poisson shot noise processes as we consider in this section.	
		
		For a stationary Poisson process $\eta$ on $\mathbb{R}^d$ with intensity measure $\lambda_d$ and an integrable function $g:\mathbb{R}^d\to\mathbb{R}$ let 
		\begin{align}\label{equation:Poisson shot noise}
			f_{\eta}(x)=\sum_{y\in\eta}g(x-y)
		\end{align}
		for $x\in\mathbb{R}^d$. We denote $(f_\eta(x))_{x\in\mathbb{R}^d}$ as Poisson shot noise process and note that it is translation invariant. Its excursion set at level $u>0$ consists of all $x\in\mathbb{R}^d$ such that $f_\eta(x)\geq u$. The corresponding volume of the excursion set in an observation window $B^d(0,s)$ with $s\geq 1$ is given by
		\begin{align*}
			F_s=\lambda_d(\{x\in B^d(0,s):	f_{\eta}(x)\geq u\}).
		\end{align*}
Now one is interested in the behaviour of $F_s$ as $s\to\infty$, i.e.\ if the observation window is increased. In \cite{BST12} variance asymptotics and central limit theorems for the volume of excursion sets of quasi-associated random fields were considered, which include a large class of Poisson shot noise processes (see \cite[Proposition 1]{BST12}). More recently, asymptotics for the variance and central limit theorems for the volume, the perimeter and the Euler characteristic of the excursion sets of Poisson shot-noise processes were shown in \cite[Section 4]{L19}, while the paper \cite{LPY20} studied the same questions for smoothed versions of volume and perimeter.	
		
We use the following assumption on the kernel function $g$.
		\begin{assu}
			\label{assumption:g}
			 There exist constants $\underline{c}_g,\overline{c}_g,\delta,\gamma>0$ and $c_g\geq 1$ such that $\delta+d/2>\gamma\geq\delta>3d$ and
			 \begin{align*}
			 	 \underline{c}_g\lVert x \rVert^{-\gamma}\leq \lvert g(x)\rvert\leq \overline{c}_g\lVert x\rVert^{-\delta} 
			 \end{align*}
		 for all $x\in \mathbb{R}^d$ with $\lVert x\rVert\geq c_g$.
		\end{assu}

By using our Theorem \ref{thm:varbound}, we derive lower bounds for variances, which complement the findings from \cite{BST12,L19}; see the discussion below for more details.	
		
		\begin{theorem}
			\label{theorem:PSN process}
		Let $g:\mathbb{R}^d\to\mathbb{R}$ be a continuous function with $g(0)>0$.
			\begin{enumerate}
				\item[a)] If $g$ fulfils Assumption \ref{assumption:g}, there exists a constant $c>0$ such that
				\begin{align*}
					\mathrm{Var}[F_s]\geq cs^d
				\end{align*}
				for $s\geq 1$. \label{poisson_shot_noise_a}
				\item[b)] Assume that $g$ has compact support $S$.
				Then, there exists a constant $c>0$ such that
				\begin{align*}
					\mathrm{Var}[F_s]\geq cs^d
				\end{align*}
				for $s\geq 1$.
			\end{enumerate}
		\end{theorem}

Replacing $g$ by $g(\cdot-z)$ for any $z\in\mathbb{R}^d$ leads to a translation of the Poisson shot noise field and, thus, by translation invariance, to a Poisson shot noise process with the same distribution. Thus, the assumption $g(0)>0$ is no loss of generality because any $g$ that can take positive values can be modified accordingly, while the case of a non-positive function $g$ is trivial because then the level set for $u>0$ becomes empty.

Since the volume of the excursion set can be written as integral over indicator functions, one obtains with Fubini's theorem and translation invariance of the Poisson shot noise process
\begin{align*}
\mathrm{Var}[F_s] & = \E\left[ \left( \int_{B^d(0,s)} \mathbf{1}\{f_\eta(x)\geq u\} \,\mathrm{d}x  \right)^2 \right] - \E\left[ \int_{B^d(0,s)} \mathbf{1}\{f_\eta(x)\geq u\} \,\mathrm{d}x \right]^2 \\
& = \int_{B^d(0,s)} \int_{B^d(0,s)} \p( f_\eta(x_1)\geq u, f_\eta(x_2)\geq u ) - \p( f_\eta(x_1)\geq u) \p(f_\eta(x_2)\geq u )  \; \mathrm{d}x_1 \; \mathrm{d}x_2 \\
& = \int_{\mathbb{R}^d} \lambda_d(\{ y\in\mathbb{R}^d: y,y+z\in B^d(0,s) \}) \\
& \quad \quad \quad \times \left( \p( f_\eta(0)\geq u, f_\eta(z)\geq u ) - \p( f_\eta(0)\geq u) \p(f_\eta(z)\geq u ) \right) \; \mathrm{d}z.
\end{align*}
Note that $\lambda_d(\{ y\in\mathbb{R}^d: y,y+z\in B^d(0,s) \})/\lambda_d(B^d(0,s))\leq 1$ for all $z\in\mathbb{R}^d$ and that it converges to one as $s\to\infty$ for all $z\in\mathbb{R}^d$. Thus, the dominated convergence theorem yields
$$
\lim_{s\to\infty} \frac{\mathrm{Var}[F_s]}{\lambda_d(B^d(0,s))} = \int_{\mathbb{R}^d} \p( f_\eta(0)\geq u, f_\eta(z)\geq u ) - \p( f_\eta(0)\geq u) \p(f_\eta(z)\geq u ) \; \mathrm{d}z
$$
if the integral on the right-hand side is well-defined. However, this explicit formula for the asymptotic variance does not imply the statement of Theorem \ref{theorem:PSN process} since the difference under the integral could take both negative and positive values in such a way that the integral becomes zero.

Since statements of the form that the variance is at least of the order of the volume of the observation window as in Theorem \ref{theorem:PSN process} were already proven in \cite[Proposition 1]{BST12} and \cite[Theorem 4.1]{L19}, let us compare the assumptions of Theorem \ref{theorem:PSN process} a) with those made before. In \cite[Proposition 1]{BST12}, it is required that $g$ is a bounded and uniformly continuous function on $\mathbb{R}^d$ with $\lvert g(x)\rvert\leq c\lVert x\rVert^\alpha$ for some constant $c>0$ and $\alpha>3d$ (as in our Assumption \ref{assumption:g}). A crucial difference is that we allow $g$ to take positive and negative values, while it has to be non-negative in \cite{BST12}, where this assumption might be essential since it ensures that the Poisson shot noise process is positively associated. A lower bound on the decay of $|g|$ as in Assumption \ref{assumption:g} is not present in \cite{BST12}, but we use it only to ensure the boundedness of the density of $f_\eta(0)$, which is supposed in \cite{BST12}. The result in \cite{BST12} deals with marks in the sense that in \eqref{equation:Poisson shot noise} each summand is multiplied by an i.i.d.\ copy of a non-negative random variable. It might be possible to generalise our results in this direction as well. The assumptions in \cite[Theorem 4.1]{L19} seem to be more restrictive than in our case. So it is supposed that $g$ depends only on the norm of its argument and that $|g(x)|$ has an upper bound as in Assumption $1$ but with $\delta=11d$. Instead a lower bound on $|g|$, a rather technical assumption (see (4.3) in \cite{L19}) is made, which even requires differentiability of $g$. We are not aware of any results dealing with the situation of part b) of Theorem \ref{theorem:PSN process}. The compact support implies that $f_\eta(0)$ does not possess a density.
We prepare the proof of Theorem \ref{theorem:PSN process} with the following lemma.

	\begin{lemma}
		\label{lemma_density}
		Let $g:\mathbb{R}^d\to\mathbb{R}$ be a continuous, bounded function with $g(0)>0$ that fulfils Assumption \ref{assumption:g}. Then, $f_\eta(x)$ has a bounded density for $x\in \mathbb{R}^d$.
	\end{lemma}
	\begin{proof}
		We use the fact that $f_\eta(x)$ has a bounded density if its characteristic function $\varphi$ is integrable.
		By \cite[Chapter 1, Lemma 3.7]{BS07} the characteristic function of $f_\eta(x)$ is given by
		\begin{align*}
			\varphi(t)=\exp\left[-\int_{\mathbb{R}^d}1-e^{\mathbf{i}tg(x-y)}\;\mathrm{d}y\right],
		\end{align*}
		where $\mathbf{i}$ is the imaginary unit. Thus, $f_\eta(x)$ has a bounded density if 
		\begin{align*}
			\int_\mathbb{R}\lvert\varphi(t)\rvert\;\mathrm{d}t=\int_\mathbb{R}\Big\lvert\exp\Big[-\int_{\mathbb{R}^d} 1-e^{\mathbf{i}tg(x-y)} \;\mathrm{d}y\Big]\Big\rvert\;\mathrm{d}t<\infty.	
		\end{align*}
		Choose $c>0$ small enough such that $1-\cos(\hat{x})=\sum_{k=1}^{\infty}(-1)^{k+1}\frac{{\hat{x}}^{2k}}{(2k)!}\geq \frac{\hat{x}^2}{4}$ for $\hat{x}\in[-c,c]$. Then it holds
		
		\begin{align*}
			\int_{\mathbb{R}^d}1-\cos(tg(x-y))\;\mathrm{d}y
			&\geq\int_{\{z\in\mathbb{R}^d: t^2g(x-z)^2\leq c^2,\lVert x-z\rVert\geq c_g\}}\frac{(tg(x-y))^{2}}{4}\;\mathrm{d}y\\
			&\geq \int_{\{z\in\mathbb{R}^d: t^2\overline{c}_g^2\lVert x-z\rVert^{-2\delta}\leq c^2,\lVert x-z\rVert\geq c_g\}}\frac{t^2\underline{c}_g^2\lVert x-y\rVert^{-2\gamma}}{4}\;\mathrm{d}y\\
			&\geq \frac{d\kappa_dt^2\underline{c}_g^2}{4}\int_{\max\left\{\left(t\overline{c}_g/c\right)^{1/\delta},c_g\right\}}^\infty r^{-2\gamma}r^{d-1}\;\mathrm{d}r\\&=\frac{d\kappa_dt^2\underline{c}_g^2}{4(2\gamma-d)}\cdot\max\left\{\left(t\overline{c}_g/c\right)^{1/\delta},c_g\right\}^{(d-2\gamma)}
		\end{align*}
		and, therefore,
		\begin{align*}
			\int_\mathbb{R}\lvert\varphi(t)\rvert\;\mathrm{d}t&=\int_\mathbb{R}\Big\lvert\exp\Big[-\int_{\mathbb{R}^d}1-e^{\mathbf{i}tg(x-y)}\;\mathrm{d}y\Big]\Big\rvert\;\mathrm{d}t\\&=2	\int_{\mathbb{R}_+}\exp\Big[-\int_{\mathbb{R}^d}1-\cos(tg(x-y))\;\mathrm{d}y\Big]\;\mathrm{d}t\\	
			&\leq 2\int_{\mathbb{R}_+}\exp\left[-\frac{d\kappa_dt^2\underline{c}_g^2}{4(2\gamma-d)}\cdot\max\left\{\left(t\overline{c}_g/c\right)^{1/\delta},c_g\right\}^{(d-2\gamma)}\right]\;\mathrm{d}t\\
			&= 2 \int_{0}^{c_g^{\delta}c/{\overline{c}_g}}\exp[-c_{1,\gamma,\delta,d} t^{2}]\;\mathrm{d}t+ 2\int_{c_g^{\delta}c/{\overline{c}_g}}^\infty\exp[-c_{2,\gamma,\delta,d}t^{(2(\delta-\gamma)+d)/\delta}]\;\mathrm{d}t\\&<\infty
		\end{align*}
		with suitable constants $c_{1,\gamma,\delta,d},c_{2,\gamma,\delta,d}>0$ since $\delta-\gamma+d/2>0$. This shows that $f_x(\eta)$ has a bounded density.
	\end{proof}
		\begin{proof}[Proof of Theorem \ref{theorem:PSN process}]
			Since $g$ is continuous and $g(x)\to0$ as $\lVert x\rVert\to\infty$, there exists a ball $B^d(\hat{t},r)$ with centre $\hat{t}\in \mathbb{R}^d$ and radius $r>0$ such that $g(t)\in[c_1,c_2]$ for all $t\in B^d(\hat{t},r)$ with
			$0<c_1<c_2<g(0)$.
			For $z\in \mathbb{R}^d$ we shall consider $D_zF_s$. The following inequalities are independent from the choice of $z$. Let $\varepsilon\in(0,\min\{1,r\})$ be small enough such that $g(x-z)\geq g(0)-c_1$ for all $x\in B^d(z,\varepsilon)$.
			Then, for $y\in B^d(z-\hat{t},r-\varepsilon)$ and $x\in B^d(z,\varepsilon)$ it holds 
			\begin{align}\label{x-y in B(t,r)}
				\lVert x-y-\hat{t}\rVert\leq \lVert x-z-y-\hat{t}+z\rVert\leq\lVert x-z\rVert+ \lVert y-(z-\hat{t})\rVert\leq \varepsilon +r-\varepsilon=r
			\end{align}
			and therefore $x-y\in B^d(\hat{t},r)$.
			
			For a) define $B_R=B^d(z,R)$ for $R\geq c_g+\varepsilon$. Let $c_3\in(0,u)$.
			Then,  for $x\in B^d(z,\varepsilon)$ the Markov inequality and the Mecke equation lead to
			\begin{align*}
				\p\Bigg(\sum_{y\in\eta\cap B_R^c}\lvert g(x-y)\rvert\geq c_3\Bigg)&\leq \frac{1}{c_3}\mathbb{E}\Bigg[\sum_{y\in\eta\cap B_R^c}\lvert g(x-y)\rvert\Bigg]\\&\leq\frac{1}{c_3}\int_{\mathbb{R}^d\backslash B_R} \lvert g(x-y)\rvert\;\mathrm{d}y\\&\leq \frac{d\kappa_d}{c_3}\int_{R-\varepsilon}^\infty \overline{c}_gr^{-\delta}r^{d-1}\;\mathrm{d}r =\frac{\overline{c}_gd\kappa_d}{c_3(\delta-d)(R-\varepsilon)^{\delta-d}}.
			\end{align*}
			Now, choose $R\geq c_g+\varepsilon$ large enough such that $B^d(z-\hat{t},r-\varepsilon)\subset B_R=B^d(z,R)$ and $	\p\Big(\sum_{y\in\eta\cap B_R^c}\lvert g(x-y)\rvert\geq c_3\Big)\leq\frac{\overline{c}_gd\kappa_d}{c_3(\delta-d)(R-\varepsilon)^{\delta-d}}\leq\frac{1}{2}$.
			This implies
			\allowdisplaybreaks
			\begin{align*}
				&\p(	f_{\eta}(x)<u, f_{\eta\cup\{z\}}(x)\geq u \;\;\forall x\in B^d(z,\varepsilon))\\&=\p\Big(u-g(x-z)\leq\sum_{y\in\eta}g(x-y)<u \;\;\forall x\in B^d(z,\varepsilon)\Big)\\
				&\geq \p\Big(u-g(0)+c_1\leq\sum_{y\in\eta}g(x-y)<u \;\;\forall x\in B^d(z,\varepsilon)\Big)\\
				&\geq \p\Big(u-g(0)+c_1-\sum_{y\in\eta\cap B_R^c}g(x-y)\leq\sum_{y\in\eta\cap B_R}g(x-y),\\&\;\;\;\;\;\;\;\;\sum_{y\in\eta\cap B_R}g(x-y)<u-\sum_{y\in\eta\cap B_R^c}g(x-y) \;\;\forall x\in B^d(z,\varepsilon)\Big)\\
				&\geq \p\Big(u-g(0)+c_1+\sum_{y\in\eta\cap B_R^c}\lvert g(x-y)\rvert\leq\sum_{y\in\eta\cap B_R}g(x-y),\\&\;\;\;\;\;\;\;\;\sum_{y\in\eta\cap B_R}g(x-y)<u-\sum_{y\in\eta\cap B_R^c}\lvert g(x-y)\rvert \;\;\forall x\in B^d(z,\varepsilon)\Big)\\
				&\geq \frac{1}{2}\p\Big(u-g(0)+c_1+c_3\leq\sum_{y\in\eta\cap B_R}g(x-y)< u-c_3 \;\;\forall x\in B^d(z,\varepsilon)\Big),
			\end{align*}
		where the last inequality follows from the independence of $\eta\cap B_R$ and $\eta\cap B_R^c$.
			Now, choose $c_1$ close enough to $c_2$ and $c_1,c_2,c_3>0$ small enough such that 
			\begin{align*}
				u-c_3>\frac{c_2}{c_1}(u-g(0)+2c_1+c_3).
			\end{align*}
			Let $k\in\mathbb{N}_0$ be such that $\frac{u-g(0)+c_1+c_3}{c_1}\leq k< \frac{u-c_3}{c_2}$. Note that such a $k$ exists because $c_1,c_2$ and $c_3$ are chosen in such a way that $\frac{u-g(0)+c_1+c_3}{c_1}+1=\frac{u-g(0)+2c_1+c_3}{c_1}<\frac{u-c_3}{c_2}$ and $\frac{u-c_3}{c_2}>0$. Therefore, together with \eqref{x-y in B(t,r)},
			\begin{align*}
				&\p(	f_{\eta}(x)<u, f_{\eta\cup\{z\}}(x)\geq u \;\;\forall x\in B^d(z,\varepsilon))\\
				&\geq \frac{1}{2}\p\left(kc_1\leq\sum_{y\in\eta\cap B_R}g(x-y)\leq kc_2 \;\;\forall x\in B^d(z,\varepsilon)\right)\\
				&\geq \frac{1}{2}\p(\eta(B^d(z-\hat{t},r-\varepsilon))=k,\eta(B_R\backslash B^d(z-\hat{t},r-\varepsilon))=0)=:p_a>0.
			\end{align*}	

			For b) we define $\tilde{S}=\{y\in \mathbb{R}^d:x-y\in S\; \text{for some } x\in B^d(z,\varepsilon)\}$. Note that $B^d(z-\hat{t},r-\varepsilon)\subseteq\tilde{S}$ because $x-y\in B^d(\hat{t},r)\subseteq S$ for all $x\in B^d(z,\varepsilon)$ and $y\in B^d(z-\hat{t},r-\varepsilon)$. Then it follows
			\begin{align*}
				&\p(	f_{\eta}(x)<u, f_{\eta\cup\{z\}}(x)\geq u \;\;\forall x\in B^d(z,\varepsilon))\\&=\p\left(u-g(x-z)\leq\sum_{y\in\eta}g(x-y)<u\;\;\forall x\in B^d(z,\varepsilon)\right)\\
				&\geq \p\left(u-g(0)+c_1\leq\sum_{y\in\eta}g(x-y)<u\;\;\forall x\in B^d(z,\varepsilon)\right)\\
				&=\p\left(u-g(0)+c_1\leq\sum_{y\in\eta\cap \tilde{S}}g(x-y)< u\;\;\forall x\in B^d(z,\varepsilon)\right).
			\end{align*}
			Now, choose $c_1$ close enough to $c_2$ and $c_1,c_2>0$ small enough such that 	
			\begin{align*}
				u>\frac{c_2}{c_1}(u-g(0)+2c_1).
			\end{align*}
			Let $k\in\mathbb{N}_0$ be such that $\frac{u-g(0)+c_1}{c_1}\leq k<\frac{u}{c_2}$. Note that such a $k$ exists because $c_1$ and $c_2$ are chosen in such a way that $\frac{u-g(0)+c_1}{c_1}+1=\frac{u-g(0)+2c_1}{c_1}<\frac{u}{c_2}$ and $\frac{u}{c_2}>0$. Then,
			\begin{align*}
				&\p(	f_{\eta}(x)<u, f_{\eta\cup\{z\}}(x)\geq u \;\;\forall x\in B^d(z,\varepsilon))\\
				&\geq \p\left(kc_1\leq\sum_{y\in\eta\cap \tilde{S}}g(x-y)\leq kc_2\;\;\forall x\in B^d(z,\varepsilon)\right)\\
				&\geq \p(\eta( B^d(z-\hat{t},r-\varepsilon))=k,\eta(\tilde{S}\backslash  B^d(z-\hat{t},r-\varepsilon))=0)=:p_b>0.
			\end{align*}
			Altogether, for $A_s=\{z\in \mathbb{R}^d:B^d(z,\varepsilon)\subset B^d(0,s)\}$ and $p=p_a$ in case of a) or $p=p_b$ in case of b) we conclude that
			\begin{align*}
				\E\left[\int(D_zF_s)^2\;\mathrm{d}z\right]&\geq \kappa_{d}^2\varepsilon^{2d}\int_{\mathbb{R}^d}\p(D_zF_s\geq\kappa_d\varepsilon^d)\;\mathrm{d}z\\&\geq \kappa_{d}^2\varepsilon^{2d}\int_{A_s}\p(	f_{\eta}(x)<u, f_{\eta\cup\{z\}}(x)\geq u \;\;\forall x\in B^d(z,\varepsilon))\;\mathrm{d}z\\
				&\geq \kappa_{d}^2\varepsilon^{2d}\int_{A_s}p\;\mathrm{d}z\geq \kappa_{d}^2\varepsilon^{2d}p\lambda_d(A_s)\geq c_{d,\varepsilon} s^d
			\end{align*}
			for some constant $c_{d,\varepsilon}>0$.
			
			In the following we consider the second-order difference operator to check \eqref{condition}. For $z_1,z_2\in\mathbb{R}^d$ with $z_1\neq z_2$ we have
			$$
			D^2_{z_1,z_2}F_s = \int_{B^d(0,s)} D^2_{z_1,z_2}\mathbf{1}\{f_\eta(x)\geq u\} \;\mathrm{d}x
			$$
			so that
			\begin{equation}\label{eqn:D2_F_s}
			|D^2_{z_1,z_2}F_s| \leq 2 \lambda_d(B_s(z_1,z_2))
			\end{equation}
			with $B_s(z_1,z_2)=\{x\in B^d(0,s): D^2_{z_1,z_2}\mathbf{1}\{f_\eta(x)\geq u\}\neq 0 \}$, where we used the bound $|D^2_{z_1,z_2}\mathbf{1}\{f_\eta(x)\geq u\}|\leq 2$. The inequality \eqref{eqn:D2_F_s} leads to
			$$
			I:=\E\left[\int_{\mathbb{R}^d}\int_{\mathbb{R}^d}(D^2_{z_1,z_2}F_s)^2\;\mathrm{d}z_1\;\mathrm{d}z_2\right] \leq 4 \int_{\mathbb{R}^d}\int_{\mathbb{R}^d}\E\left[\lambda_d(B_s(z_1,z_2))^2\right]\;\mathrm{d}z_1\;\mathrm{d}z_2.
			$$
			
			First we study the situation of a). Let $x\in B^d(0,s)$ and assume that $|g(x-z_2)|\leq |g(x-z_1)|$. Since
			\begin{align*}
			D^2_{z_1,z_2}\mathbf{1}\{f_\eta(x)\geq u\} & = \mathbf{1}\{f_\eta(x)+g(x-z_1)+g(x-z_2)\geq u\} - \mathbf{1}\{f_\eta(x)+g(x-z_1)\geq u\} \\
			& \quad - ( \mathbf{1}\{f_\eta(x)+g(x-z_2)\geq u\} - \mathbf{1}\{f_\eta(x)\geq u\} ),
			\end{align*}
			we obtain that
			$$
			D^2_{z_1,z_2}\mathbf{1}\{f_\eta(x)\geq u\} = 0
			$$
			if
			$$
			f_\eta(x)+g(x-z_1) \notin [u-|g(x-z_2)|,u+|g(x-z_2)|]
			$$
			and
			$$
			f_\eta(x) \notin [u-|g(x-z_2)|,u+|g(x-z_2)|].
			$$
			Together with the fact that the density of $f_\eta(x)$ is bounded by a constant $C_1>0$, which was shown in Lemma \ref{lemma_density}, we derive
			\begin{align*}
			\p(x\in B_s(z_1,z_2)) & \leq \p( f_\eta(x)+g(x-z_1) \in [u-|g(x-z_2)|,u+|g(x-z_2)|] ) \\
			& \quad + \p(f_\eta(x) \in [u-|g(x-z_2)|,u+|g(x-z_2)|]) \\
			& \leq 4 C_1 |g(x-z_2)|.
			\end{align*}
			Using the same arguments for $|g(x-z_2)|\geq |g(x-z_1)|$, we deduce
			$$
			\p(x\in B_s(z_1,z_2)) \leq 4 C_1 \min\{|g(x-z_1)|,|g(x-z_2)|\}
			$$
			so that with H\"older's inequality and the inequality $\min\{a,b\}\leq \sqrt{a}\sqrt{b}$ for $a,b\geq 0$,
			\begin{align*}
				& \E\left[\lambda_d(B_s(z_1,z_2))^2\right] \\
				&=\int_{B^d(0,s)}\int_{B^d(0,s)}\p(x_1\in B_s(z_1,z_2),x_2\in B_s(z_1,z_2))\;\mathrm{d}x_1\;\mathrm{d}x_2\\&\leq \int_{B^d(0,s)}\int_{B^d(0,s)} \p(x_1\in B_s(z_1,z_2))^{2/3}\p(x_2\in B_s(z_1,z_2))^{1/3}\;\mathrm{d}x_1\;\mathrm{d}x_2\\
				&\leq 4C_1 \int_{B^d(0,s)}\int_{B^d(0,s)}\lvert g(x_1-z_1)\rvert^{1/3}\lvert g(x_1-z_2)\rvert^{1/3}\lvert g(x_2-z_1)\rvert^{1/3}\;\mathrm{d}x_1\;\mathrm{d}x_2.
			\end{align*}
			From Assumption \ref{assumption:g} and the continuity of $g$ it follows that $g$ is bounded by a constant $C_2>0$. Using the decay of $|g|$ and $\delta>3d$ in Assumption \ref{assumption:g}, we have for $x\in B^d(0,s)$ that
			\begin{align*}
			\int_{\mathbb{R}^d} \lvert g(x-z)\rvert^{1/3}\;\mathrm{d}z&=	\int_{\mathbb{R}^d\backslash B^d(x,c_g)} \lvert g(x-z)\rvert^{1/3}\;\mathrm{d}z+\int_{B^d(x,c_g)} \lvert g(x-z)\rvert^{1/3}\;\mathrm{d}z\\
			&\leq \int_{\mathbb{R}^d\backslash B^d(x,c_g)} \overline{c}_g^{1/3}\lVert x-z\rVert^{-\delta/3}\;\mathrm{d}z+C_2^{1/3}\kappa_{d}c_g^d\\
			&= d\kappa_{d}\overline{c}_g^{1/3}\int_{c_g}^\infty r^{d-1}r^{-\delta/3}\;\mathrm{d}r+C_2^{1/3}\kappa_{d}c_g^d\\
			&=d\kappa_{d}\overline{c}_g^{1/3}\frac{c_g^{d-\delta/3}}{\delta/3-d}+C_2^{1/3}\kappa_{d}c_g^d=:C_3.
			\end{align*}
						The same estimate holds for $\int_{B^d(0,s)} \lvert g(x-z)\rvert^{1/3}\;\mathrm{d}x$ for $z\in\mathbb{R}^d$. Hence,
			\begin{align*}
			I&\leq \int_{\mathbb{R}^d}\int_{\mathbb{R}^d}16 C_1 \int_{B^d(0,s)}\int_{B^d(0,s)}\lvert g(x_1-z_1)\rvert^{1/3}\lvert g(x_1-z_2)\rvert^{1/3}\\&\hspace{6.5cm}\times\lvert g(x_2-z_1)\rvert^{1/3}\;\mathrm{d}x_1\;\mathrm{d}x_2\;\mathrm{d}z_1\;\mathrm{d}z_2\\
				&=16 C_1 \int_{B^d(0,s)}\int_{\mathbb{R}^d}\lvert g(x_1-z_1)\rvert^{1/3}\int_{B^d(0,s)}\lvert g(x_2-z_1)\rvert^{1/3}\\&\hspace{6.5cm}\times\int_{\mathbb{R}^d}\lvert g(x_1-z_2)\rvert^{1/3} 
				\;\mathrm{d}z_2\;\mathrm{d}x_2\;\mathrm{d}z_1\;\mathrm{d}x_1\\
				&\leq 16 C_1 \int_{B^d(0,s)}C_3^3\;\mathrm{d}x_1=: \tilde{c}_1s^d.
			\end{align*}
			
			For b) let $\widetilde{R}>0$ be such that $S\subseteq B^d(0,\widetilde{R})$ and let $z_1,z_2\in\mathbb{R}^d$. Then, since
				\begin{align*}
					B_s(z_1,z_2) \subseteq \{x\in B^d(0,s):\lVert x-z_1\rVert\leq \widetilde{R},\lVert x-z_2\rVert\leq \widetilde{R}\},
				\end{align*}
			it follows
			\begin{align*}
			\E\left[\lambda_d(B_s(z_1,z_2))^2\right] \leq \lambda_d(\{x\in B^d(0,s):\lVert x-z_1\rVert\leq \widetilde{R},\lVert x-z_2\lVert\leq \widetilde{R}\})^2.
			\end{align*}
			The triangle inequality implies $\lambda_d(\{x\in B^d(0,s):\lVert x-z_1\rVert\leq \widetilde{R},\lVert x-z_2\lVert\leq \widetilde{R}\})=0$ for $\lVert z_1-z_2\rVert>2\widetilde{R}$ or $\lVert z_2\rVert> s+\widetilde{R}$ and therefore
			\begin{align*}
			I &\leq 4 \int_{B^d(0,s+\widetilde{R})}\int_{B^d(z_2,2\widetilde{R})} \lambda_d(\{x\in B^d(0,s):\lVert x-z_1\rVert\leq \widetilde{R},\lVert x-z_2\lVert\leq \widetilde{R}\})^2\;\mathrm{d}z_1\;\mathrm{d}z_2\\&\leq 4\int_{B^d(0,s+\widetilde{R})}\int_{B^d(z_2,2\widetilde{R})} (\kappa_d\widetilde{R}^d)^2\;\mathrm{d}z_1\;\mathrm{d}z_2\leq 4(\kappa_d\widetilde{R}^d)^2\kappa_d^2(2\widetilde{R})^d(s+\widetilde{R})^d\\&\leq Cs^d
			\end{align*}
		for a suitable constant $C>0$, which completes together with Theorem \ref{thm:varbound} the proof.
		\end{proof}
		
\appendix		
		
	\section{Appendix on stabilising functionals}
	\label{appendix:stabilising_functionals}
	
In this appendix we recall the framework of stabilising functionals considered in \cite{LSY19,SY21}. For further works on stabilisation in stochastic geometry we refer the reader to e.g.\ \cite{BY05,LPY20,P07,PW08,PY01,PY05} and the references therein. Let $(\mathbb{X},\mathscr{F}_\mathbb{X})$ be a measurable space with a $\sigma$-finite measure $\hat{\lambda}$ and a measurable semi-metric ${\rm d}$. We denote by $B(x,r)$ the ball of radius $r$ with respect to ${\rm d}$ around $x\in\mathbb{X}$ and assume that there exist constants $\kappa,\gamma>0$ such that
	\begin{align}
		\label{gamma}
		\limsup\limits_{\varepsilon\to0}\frac{\hat{\lambda}(B(x,r+\varepsilon))-\hat{\lambda}(B(x,r))}{\varepsilon}\leq \kappa\gamma r^{\gamma-1}
	\end{align}
for all $r\geq 0$ and $x\in\mathbb{X}$. Obviously this assumption is satisfied if $\mathbb{X}$ is $\mathbb{R}^d$ or a subset of $\mathbb{R}^d$ equipped with the usual Euclidean norm and $\hat{\lambda}$	has a bounded density with respect to the Lebesgue measure.

For $s\geq 1$ let $\eta_s$ be a Poisson process with intensity measure $s\hat{\lambda}$. We consider a Poisson functional $F_s$, i.e.\ a random variable that depends on the Poisson process $\eta_s$. In many applications $F_s$ can be written as a sum of scores $\xi_s$, i.e.\ 
	\begin{align}\label{eqn:sum_scores}
		F_s=F_s(\eta_s)=\sum_{x\in\eta_s}\xi_s(x,\eta_s).
	\end{align}
Here, one can think of $F_s$ as the sum of contributions associated with the points of $\eta_s$. In the sequel, we assume that the scores are stabilising. Here the idea is that the score of a point $x$ only depends on the points of $\eta_s$ in a random neighbourhood of $x$.
	
	In order to show the condition \eqref{condition} for random variables of the form \eqref{eqn:sum_scores}, one can often use properties of the score functions. The following definitions were taken from \cite{LSY19,SY21}.
	We start with defining the radius of stabilisation.
	Let $s\geq 1$. A measurable map $R_s:\mathbb{X}\times \mathbf{N}\to\mathbb{R}$ is called radius of stabilisation for $\xi_s$ if
		\begin{align*}
			\xi_s(x,(\nu\cup\{x\}\cup A)\cap B^d(x,R_s(x,\nu\cup\{x\})))=\xi_s(x,\nu \cup\{x\}\cup A)
		\end{align*}
		for all $x\in\mathbb{X}$, $\nu\in\mathbf{N}$ and $A\subset\mathbb{X}$ with $\lvert A\rvert\leq 9$. Broadly speaking, this says that the value of the score only depends on the points of the underlying point configuration with distance at most $R_s(x,\nu\cup\{x\})$ from $x$.
		Using this radius of stabilisation, one can define exponential stabilisation.
		The scores $(\xi_s)_{s\geq 1}$ are called exponentially stabilising if there exist radii of stabilisation and constants $C_{stab},c_{stab},\alpha_{stab}>0$ such that
		\begin{align*}
			\p(R_s(x,\eta_s\cup\{x\})\geq r)\leq C_{stab}\exp[-c_{stab}(s^{1/\gamma}r)^{\alpha_{stab}}]
		\end{align*}
		for $x\in\mathbb{X}$, $r\geq 0$, $s\geq 1$ and $\gamma$ from \eqref{gamma}.
	For $q>0$, the scores $(\xi_s)_{s\geq 1}$ fulfil a $q$-th moment condition if there exists a constant $C_q>0$ satisfying
		\begin{align*}
			\sup\limits_{s\geq 1}\sup\limits_{x\in\mathbb{X}}\E\lvert\xi_s(x,\eta_s\cup\{x\}\cup A )\rvert^q\leq C_q
		\end{align*}
		for $A\subset\mathbb{X}$ with $\lvert A\rvert\leq 9$.
	Finally, the scores $(\xi_s)_{s\geq 1}$ decay exponentially fast with distance to a measurable set $K\subseteq\mathbb{X}$ if there are constants $C_K,c_K,\alpha_K>0$ such that for $x\in\mathbb{X}$, $r\geq 0$, $s\geq 1$ and $A\subset\mathbb{X}$ with $\lvert A\rvert\leq 9$,
	\begin{align*}
		\p(\xi_s(x,\eta_s\cup\{x\}\cup A)\neq 0)\leq C_K\exp[-c_Ks^{\alpha_K/\gamma}{\rm d}(x,K)^{\alpha_K}],
	\end{align*}
	where ${\rm d}(x,K)$ denotes the distance from $x$ to $K$ with respect to the semi-metric ${\rm d}$ and $\gamma$ is from \eqref{gamma}. In contrast to the definitions in \cite{LSY19}, those in \cite{SY21} and in this appendix \ref{appendix:stabilising_functionals} require that one can add up to nine additional points instead of seven, but this difference is not essential and all results from \cite{LSY19} we refer to throughout this paper are still valid. For more details on stabilising functionals we refer to \cite{LSY19} or \cite{SY21} and the references therein.

\end{document}